\documentclass[reqno,10pt]{article} %definisco il tipo di documento, la grandezza del font e la pagina
\usepackage[T1]{fontenc} %codifica font (lingue occidentali)
\usepackage[utf8]{inputenc} %scrittura e interpretazione dei caratteri nell'editor. Al posto di latin1 si scrive pure utf8. consente le lettere accentate
\usepackage{lmodern}
\usepackage{microtype} 
\usepackage{lipsum}
\usepackage{titling}		
\usepackage[italian,english]{babel} %lingua principale inglese e secondaria italiano. Per inserire brevi testi in altre lingue pag 20 de "l'arte di scrivere in latex"
\usepackage{indentfirst} %rientro capoversi
\usepackage{url} %scrivere siti internet
\usepackage{soul}

\usepackage{bm} %carattere in neretto in una formula ex \bm{\alpha}
\usepackage{bbm} %funzione caratteristica
\usepackage{dsfont}%funzione caratteristica
\usepackage{mathtools}
\usepackage{mathrsfs}  
\usepackage{array}
\usepackage{amsmath,amsfonts,latexsym,amsthm}
\usepackage[mathcal]{euscript}
\usepackage{bigints}

\usepackage[pdfpagelabels]{hyperref}
\hypersetup{
		colorlinks=true,
		linkcolor=blue,
		anchorcolor=black,
		citecolor=green,
		urlcolor=red}

\counterwithin{equation}{section}
\usepackage[noabbrev]{cleveref}
\usepackage[psamsfonts]{amssymb}
\usepackage{faktor} 
\usepackage[new]{old-arrows}
\usepackage[shortlabels]{enumitem}

\usepackage{float} 
\usepackage{booktabs} %tabelle
\usepackage[scriptsize]{caption}[2020/01/03] %didascalie
\usepackage{graphicx} %immagini
\usepackage{subfigure} %più figure in una figura
\usepackage{floatflt} %figure con testo accanto ma funziona meglio di wrapfig

\usepackage{pgfplots}
\pgfplotsset{width=7cm,compat=1.8}
\usetikzlibrary{angles,quotes,calc}
\usepackage{tikz}

\usepackage[autostyle]{csquotes}

\usepackage[
    backend=biber,
    style=alphabetic,
    sortlocale=de_DE,
    natbib=true,
    url=false, 
    doi=true,
    eprint=false
]{biblatex}
\usepackage{fancyhdr, graphicx}

\newtheorem{thm}{Theorem}[section]
\newtheorem{prop}[thm]{Proposition}

\newtheorem{lem}[thm]{Lemma}
\newtheorem{cor}[thm]{Corollary}

\theoremstyle{definition}
\newtheorem{dfn}[thm]{Definition}

\theoremstyle{remark}
\newtheorem{oss}[thm]{Remark}

\theoremstyle{definition}

\newtheorem{ip}[thm]{Hypothesis}

\newcommand{\iprod}[2]{\langle#1,#2\rangle}

\newcommand{\norm}[1]{\left\lVert#1\right\rVert}
\newcommand{\abs}[1]{\left\lvert#1\right\rvert}

\newcommand{\Nat}{\mathbb{N}}

\newcommand{\R}{\mathbb{R}}

\newcommand{\al}{\alpha}
\newcommand{\be}{\beta}
\newcommand{\g}{\gamma}
\newcommand{\ep}{\varepsilon}
\newcommand{\om}{\omega}
\newcommand{\si}{\sigma}

\newcommand{\ph}{\varphi}
\newcommand{\la}{\lambda}

\newcommand{\de}{\partial}
\newcommand{\sm}{\smallsetminus}
\newcommand{\op}{\mathcal{L}}

\newcommand{\Tr}[1]{{\rm Tr}\left(#1\right)}

\newcommand{\at}[1]{\biggl\rvert_{#1}}

\newcommand{\Ng}{\mathcal{N}}
\newcommand{\ac}{\mathrm{H}}

\newcommand{\uts}{U(t,s)}

\newcommand{\pst}{P_{s,t}}

\newcommand{\Qm}[1]{Q\left(#1\right)^{-\frac{1}{2}}}
\newcommand{\Qp}[1]{Q\left(#1\right)^{\frac{1}{2}}}

\newcommand{\ix}[2]{\langle#1,#2\rangle_X}

\newcommand{\os}{\mathcal{O}}
\newcommand{\oa}{\mathcal{A}}
\newcommand{\id}{{\rm Id}}

\newcommand{\scal}[2]{{\left\langle #1,#2\right\rangle}}
\newcommand{\Id}{{\operatorname{Id}}}
\usepackage{scalerel,stackengine}
\stackMath
\newcommand\rwhat[1]{%
\savestack{\tmpbox}{\stretchto{%
  \scaleto{%
    \scalerel*[\widthof{\ensuremath{#1}}]{\kern-.6pt\bigwedge\kern-.6pt}%
    {\rule[-\textheight/2]{1ex}{\textheight}}%WIDTH-LIMITED BIG WEDGE
  }{\textheight}% 
}{0.5ex}}%
\stackon[1pt]{#1}{\tmpbox}%
}
\title{Log-Sobolev inequalities and hypercontractivity for Ornstein-Uhlenbeck evolution operators in infinite dimensions}
\author{Davide A. Bignamini and Paolo De Fazio}

\begin{document}

\maketitle

\footnote{\textbf{subjclass (2020)}: 28C20, 46G12, 60H15}

\begin{abstract}
In an infinite dimensional separable Hilbert space $X$, we study the realizations of Ornstein-Uhlenbeck evolution operators $\pst$ in the spaces $L^p(X,\g_t)$, $\{\g_t\}_{t\in\R}$ being a suitable evolution system of measures for $\pst$ in $\R$. We prove hyperconctractivity results, relying on suitable Log-Sobolev estimates. Among the examples we consider the transition evolution operator associated to a non autonomous stochastic parabolic PDE.

\end{abstract}

\section{Introduction}

Let $(X, \ix{\cdot}{\cdot},\norm{\,\cdot\,}_X)$ be a separable Hilbert space and set $\Delta=\{(s,t)\in\R^2\ |\ s<t\}$. Let $\{U(t,s)\}_{(t,s)\in\overline{\Delta}}$ be an evolution operator in $X$ and let $\{B(r)\}_{r\in\R}$ be a strongly continuous family of linear bounded operators on $X$. In this paper we consider a class of evolution operators $\{P_{s,t}\}_{(s,t)\in\overline{\Delta}}$ defined on the space of bounded and Borel measurable functions $\ph$ by
\begin{align*}
 P_{r,r}&=I,\qquad \forall\; r\in\R, \\
P_{s,t}\ph(x)&=\int_X\ph(y)\,\Ng_{U(t,s)x ,Q(t,s)}(dy),\quad (s,t)\in\Delta, x\in X
\end{align*}
where $\Ng_{U(t,s)x ,Q(t,s)}$ is the Gaussian measure in $X$ with mean $\uts x$
and covariance operator
\begin{align} 
Q(t,s)=\int_s^tU(t,r) Q(r) U(t,r)^\star\,dr,\qquad Q(r):=B(r)B(r)^\star.
\end{align}
Of course we assume that $Q(t,s)$ has finite trace for every $(s,t)\in\Delta$.

The main achievement of this paper is the proof of the hypercontrivity of $\pst$ in Lebesgue spaces with respect to suitable measures. It relies on a family of logaritmic Sobolev inequalities that is the second main result of this paper.

We recall that in the autonomous case, $\pst=P_{0,t-s}$ is a semigroup and it is settled in Lebesgue spaces with respect to its invariant measure, that exists and is unique under suitable assumptions. 

In the non autonomous case a single invariant measure does not exist in general, being replaced by evolution systems of measures, namely families of Borel probability measures $\{\g_r\}_{r\in\R}$ in $X$ such that 
\begin{equation*}
\int_XP_{s,t}\ph(x)\gamma_s(dx)=\int_X\ph(x)\gamma_t(dx),
\end{equation*}
for every $s<t$ and for every bounded and continuous $\ph:X\rightarrow\R$. So, an obviuous difficulty arises, namely the spaces $L^p(X,\g_r)$ depend explicitly on $r$ and we cannot set our problem in a fixed $L^p$ space.

The starting point of our analysis are basic (but not trivial in this setting) results on the relation between $\{P_{s,t}\}_{(s,t)\in\overline{\Delta}}$ and a family of non autonomous Ornstein-Uhlenbeck type operators $\{L(r)\}_{r\in\R}$ given by 
\begin{equation*}
L(t)\ph(x)=\frac{1}{2}\mbox{\rm Tr}\Bigl(Q(t)D^2\ph(x)\Bigr)+\ix{A(t)x}{\nabla\ph(x)}.
\end{equation*} 
Here $\{A(t)\}_{t\in\R}$ is a family of linear and not necessarily bounded operators associated to $\{U(t,s)\}_{(s,t)\in\overline{\Delta}}$.  In the more significant example, $\{A(t)\}_{t\in\R}$ is a family of realizations of elliptic operators in a $L^p$ space, see Section \ref{Examples}.

In Section \ref{connessione} we prove that for every $(s,t)\in\Delta$ and $x\in X$ we have
\begin{align}
&\frac{\de}{\de s} \pst \ph (x)=-L(s)\pst \ph (x)\label{1int},\\
&\frac{\de}{\de t} \pst \ph (x)=\pst L(t) \ph (x)\label{2int},
\end{align}
where $\ph$ belongs to the space of smooth cylindrical functions $\mathcal{E}_t(X)$ defined in \eqref{trigonometriciR}.

In Section \ref{misureinvarianti} we provide conditions that guarantee existence of an evolution system of measures $\{\gamma_r\}_{r\in\R}$ where $\g_r$ is a Gaussian measure with mean zero and it satisfies
\begin{equation*}
\lim_{s\rightarrow-\infty}\pst \ph(x)=\int_X\ph(y)\, \, \g_t(dy),\quad t\in\R,\; x\in X,
\end{equation*}
for every bounded and continuous function $\ph:X\rightarrow \R$.
Such results are already contained in \cite{ouy-roc2016} (see also \cite{gei-lun2008} \cite{dap-rock2008} for the case $X=\R^n$), however we give the proofs for the convenience of the readers.

As in the finite dimensional setting, if an evolution system of measures $\{\g_r\}_{r\in\R}$ exists, then it is possible to extend each operator $\pst$ to a linear bounded operator from $L^p(X,\g_t)$ to $L^p(X,\g_s)$, denoted by $\pst^{(p)}$. Such operators are consistent, i.e. for every  $s<t$ and $f\in L^p(H, \g_t)\cap L^q(H, \g_t)$ it holds $\pst^{(p)}f=\pst^{(q)}f$. For this reason, we will omit the index $p$ if no confusion can arise, and we still denote them by $\pst$. 

In Section \ref{log-sobolev}, for the Gaussian evolution system of measures constructed in Section \ref{misureinvarianti},  we prove a family of logarithmic Sobolev inequalities,
\begin{equation}\label{log-intro}
\int_X|\ph|^p\log\left(|\ph|^p\right)\, d\g_t\leq m_t(|\ph|^p)\log\left(m_t\left(|\ph|^p\right)\right)+\kappa p^2\int_X |\ph|^{p-2}\norm{\Qp{t}\nabla \ph}^2\, \mathbbm{1}_{\{\ph\neq 0\}}d\g_t.
\end{equation}
where $\ph\in C^1_b(X)$, $t\in\R$, $p\in (1,+\infty)$, $\displaystyle{m_t(\ph):=\int_X \ph\, d\nu_t}$ and $\kappa$ is an explicit positive constant (see Theorem \ref{LOG}).

Exploiting \eqref{log-intro} we prove  a hypercontractivity result for $\{P_{s,t}\}_{(s,t)\in\overline{\Delta}}$, namely for every $(s,t)\in\overline{\Delta}$, $q>1$ and $p\leq c(s,t,q):=(q-1)e^{\frac{t-s}{2\kappa}}+1$ we have
\begin{equation}\label{iper-intro}
\norm{\pst \ph}_{L^p(X,\, \g_s)}\leq \norm{\ph}_{L^q(X,\, \g_t)},\quad \ph\in L^q(X,\, \g_t),
\end{equation}
where $\kappa$ is the same constant appearing in \eqref{log-intro} and we still denote by $\pst$ the extension of $\pst$ from $L^g(X,\gamma_t)$ to $L^g(X,\gamma_s)$ for every $g>1$.  We stress that $c(s,t,q)$ is the non autonomous version of the optimal constant in the autonomous case, see Remark \ref{ottimale}.

Finally in Section \ref{Examples} we present different examples of $\{P_{s,t}\}_{(s,t)\in\overline{\Delta}}$ that verify our assumptions. In our most significant example, $A(r)$ is the realization of a negative second order elliptic differential operator in $X=L^2(\os)$ with Dirichlet or Robin boundary conditions and smooth enough coefficients; $\os$ is a bounded open smooth subset of $\R^d$. $\{U(t,s)\}_{(s,t)\in\overline{\Delta}}$ is the evolution operator associated to $\{A(r)\}_{r\in\R}$ according to Acquistapace-Terreni \cite{MR945820,MR934508} and $B(r):=(-A(r))^{-\gamma}$ with $\gamma\geq 0$. In this case $\{P_{s,t}\}_{(s,t)\in\overline{\Delta}}$ is associated to the time inhomogeneous Markov process that is the unique mild solution of the non autonomous stochastic heat equation
\begin{equation}\label{SPDE}
dZ(t)=A(t)Z(t)dt+(-A(t))^{-\gamma}dW(t),\qquad Z(0)=x\in X,
\end{equation}
where $\{W(t)\}_{t\in\R}$ is a $X$-cylindrical Wiener process. We refer to \cite{ver-zim2008} for a study of SPDEs of the type \eqref{SPDE}.

In finite dimension evolution operators for Kolmogorov equations have already been widely investigated, see for instance \cite{add2013,add-ang-lor2017,ang-lor2014,ang-lor20142,ang-lor-lun2013,dap-lun2007,gei-lun2008,gei-lun2009,kun-lor-lun2010}. Instead, in infinite dimension, only a few results are available, see for instance \cite{cerlun,defazio22,ouy-roc2016}.

In the autonomous case, where $\pst=P_{0,t-s}$ is a semigroup, formulas similar to \eqref{1int} and \eqref{2int} are known for suitable functions $\varphi$ accordingly to the theory of weakly continuous semigroups (see \cite{MR1293091,MR4553575,MR1724248}). Currently in the non-autonomous case there is no similar theory that can be exploited. We remark that we cannot use the abstract results on evolution operators of \cite{MR934508,MR533824}, since the family of realizations of the operators $\{L(r)\}_{r\in \R}$ in spaces of bounded and continuous functions from $X$ to $\R$ does not satisfy their assumptions.
If $X=\R^n$, then \eqref{1int} and \eqref{2int} were proven in \cite{kun-lor-lun2010} for smooth functions with compact support, but in infinite dimension compactly supported functions are not relevant and must be replaced by other classes of functions such as $\mathcal{E}_t(X)$ defined in \eqref{trigonometriciR}, which is dense in $L^p(X,\g)$ for every $p\in(1,+\infty)$ and for every Borel probability measure $\g$ in X. Moreover, such spaces depend explicitly on $t$,  so it is not possible to use a technique similar to the one presented in \cite{kun-lor-lun2010} to prove \eqref{1int} and \eqref{2int}.

Still if $X=\R^n$ \eqref{log-intro} was proven in \cite{ang-lor-lun2013}. In infinite dimension a hyperboundedness result for $\{\pst\}_{(s,t)\in\overline{\Delta}}$ was proven in \cite[Thm\, 5.9]{ouy-roc2016} as a consequence of a Harnack type inequality; however the constant is not necessarily $1$ and the result is significant only under suitable assumptions which imply that $\pst$ is Strong-Feller, namely when $\pst$ maps Borel bounded functions into continuous functions.

\section{Notations}\label{notations}

If $(X, \norm{\, \cdot\, }_X)$ and $(Y, \norm{\, \cdot\, }_Y)$ are real Banach spaces we denote by $\op(X;Y)$ the space of bounded linear operators from $X$ to $Y$. If $Y=\R$ we simply write $X^\star$ instead of $\op(X;\R)$. For $k\geq 2$, $\op^{k}(X; Y)$ is the space of the $k$-linear bounded operators $T: X^k\longrightarrow Y$ endowed with the norm $$\norm{T}_{\op^{k}(X; Y)}=\sup\biggl\{\frac{\norm{T(x_1,...,x_k)}_Y}{\norm{x_1}_X\cdot\cdot\cdot\norm{x_k}_X}:\ x_1,...,x_k\in X\sm\{0\}\biggr\}.$$ If $Y=\R$ we set $\mathcal{L}^0(X):=X$, $\op(X):=\op(X; X)$ and $\op^k(X):=\op^k(X;X)$ for every $k\geq 2$.

Given $A:D(A)\subseteq X\longrightarrow X$ and $V$ be a closed subspace of $X$, we call part of $A$ in $V$ the operator $\tilde{A}$ with domain $D(\tilde{A})=\{x\in D(A)\cap V\ |\ Ax\in V\}$ such that $\tilde{A}x=Ax$ for all $x\in D(\tilde{A})$.

 By $B_b(X;Y)$ and $C_b(X;Y)$ we denote the space of bounded Borel functions from $X$ to $Y$ and the space of bounded and continuous functions from $X$ to $Y$, respectively. We endow them with the sup norm $$\norm{F}_\infty=\sup_{x\in X}\norm{F(x)}_Y.$$  If $Y=\R$, we simply write $B_b(X)$ and $C_b(X)$ instead of $B_b(X;\R)$ and $C_b(X;\R)$, respectively.

Let $F:X\longrightarrow Y$. We say that $F$ is Fréchet differentiable at $x\in X$ if there exists $T_x\in\op(X,Y)$ such that
\begin{equation}\label{yfrechet2}
\lim_{\norm{h}_{X}\rightarrow 0}\frac{\norm{F(x+h)-F(x)-T_x(h)}_Y}{\norm{h}_{X}}=0.
\end{equation}
$T_x$ is the Fréchet differential of $F$ at $x$ and we denote it by $DF(x)$. We say that $F$ is Fréchet differentiable if it is Fréchet differentiable at every $x\in X$. 
If $\ph:X\longrightarrow \R$ is Fréchet differentiable at $x\in X$, we say that $\ph$ is twice Fréchet differentiable at $x$ if $D \ph:X\longrightarrow X^\star$ is Fréchet differentiable at $x$. 
 We denote by $D^2 \ph$ the unique element of $\op^{2}(X; \R)$ such that
 \[
 D^2 \ph(x)(k,h):=(T_x k)(h),\quad h, k\in X,
\]
where $T_x$ is the operator in \eqref{yfrechet2} with $F$ replaced by $D \ph$ and $Y=X^\star$. In a similar way we define the $k$-times Fréchet differentiable functions $\ph:X\rightarrow\R$ and we denote by $D^{k} \ph: X\longrightarrow \op^{k}(X; \R)$ its $k$-Fréchet derivatives.

For every $k\in \Nat\cup\{0\}$, we set $C^0_b(X)=C_b(X)$ and for every $k\geq 1$ $C^k_{b}(X)$ is the subspace of $C_b(X)$ consisting of all functions $f:X\longrightarrow \R$ $k$-times Fréchet differentiable. We endow $C^k_b(X)$  with the norm
\begin{align*}
\norm{\ph}_{C^k_b(X)}:= \norm{\ph}_\infty+\sum_{j=1}^k\sup_{x\in X}\norm{D^j \ph(x)}_{\op^{j}(X; \R)}.
\end{align*}

Now we assume that $X$ is a separable Hilbert space equipped with the inner product $\scal{\cdot}{\cdot}_{X}$.

Let $\ph\in C^k_b(X)$. By the Riesz representation theorem,  for every $j=1,...,k$, for every $x\in X$ there are unique $S^{j}_x\in \op^{j-1}(X)$ such that
\begin{align*}
D^{j} \ph(x)(h_1,...,h_j)=\langle S^j_x(h_1,...,h_{j-1}),h_j\rangle_{X},\qquad h_1,...,h_j\in X.
\end{align*}
We set $\nabla^j\ph(x):=S^j_x$ and we call $\nabla\ph(x)$ and $\nabla^2\ph(x)$ the gradient of $\ph$ and Hessian operator of $\ph$ at $x\in X$, respectively. Moreover $$\norm{\ph}_{C^k_b(X)}= \norm{\ph}_\infty+\sum_{j=1}^k\sup_{x\in X}\norm{\nabla^j \ph(x)}_{\op^{j-1}(X)}.$$

 Let $\{e_k\}_{k\in\Nat}$ be an orthonormal basis of $X$ and let $\ph:X\rightarrow$ be a $k$-times Fréchet differentiable function. As in the finite dimensional case, for all $j=1,...,k$ we define the partial derivatives of $\ph$ of order $j$ at $x\in X$ along the directions of $\{e_k\}_{k\in\Nat}$. Moreover for all $j=1,...,k$ it can be shown that 
\begin{align*}
\frac{\de^j \ph}{\de e_{i_1},...,\de e_{i_j}}(x):=\ix{\nabla^j\ph(x)(e_{i_1},...,e_{i_{j-1}})}{e_{i_j}},\qquad  i_1,...,i_j\in\Nat.
\end{align*}
%If in addition $\ph:X\rightarrow$ is twice Fréchet differentiable function we  define the second order partial derivatives of $\ph$ at $x\in X$ as 
%\begin{align*}
%\frac{\de^2 \ph}{\de e_i\de e_j}(x):=\ix{\nabla^2\ph(x)e_i}{e_j},\qquad i,j\in\Nat.
%\end{align*}

We say that $Q\in\mathcal{L}(X)$ is non-negative (respectively negative, non-positive, positive) if for every $x\in X\sm\{0\}$
\[
\langle Qx,x\rangle_X\geq 0\ (<0,\ \geq,\ >0).
\]
Let $Q\in\mathcal{L}(X)$ be a non-negative and self-adjoint operator. We say that $Q$ is a trace class operator if
\begin{align}\label{trace_defn}
\Tr{Q}:=\sum_{n=1}^{+\infty}\langle Qe_n,e_n\rangle_X<+\infty,
\end{align}
for some (and hence for all) orthonormal basis $\{e_n\}_{n\in\Nat}$ of $X$. We recall that the trace operator, defined in \eqref{trace_defn}, is independent of the choice of the orthonormal basis.
\noindent
We denote by $\op_1(X)$ the subspace of $\op(X)$ consisting of all the self-adjoint operators having finite trace and by $\op^+_1(X)$ the subspace of $\op(X)$ consisting of all non-negative self-adjoint operators having finite trace.

Let $\mu$ be a Borel probability measure on $X$. We denote by $\widehat{\mu}$ its characteristic function defined by
\[
\widehat{\mu}(x):=\int_X e^{i\ix{x}{y}}\mu(dy).
\]
Let $Q$ be a self-adjoint non-negative trace class operator and let $m\in X$. We denote by $\mathcal{N}_{m,Q}$ the Gaussian measure in $X$ with mean $m$ and covariance operator $Q$. We recall that
\begin{equation*}
\widehat{\mathcal{N}}_{m,Q}(x):=\int_X e^{i\ix{x}{y}}\mathcal{N}_{m,Q}(dy)=e^{i\ix{m}{x}-\frac{1}{2}\ix{Qx}{x}},\quad x\in X.
\end{equation*}

\subsection{Pseudo-inverse and differentiability along subspaces}\label{HR}
Let $(X, \ix{\cdot}{\cdot},\norm{\,\cdot\,}_X)$ be a separable Hilbert space.
Let $R\in\mathcal{L}(X)$ be a self-adjoint operator. We denote by $\ker R$ the kernel of $R$ and by $(\ker R)^{\bot}$ its orthogonal subspace in $H$. 

We denote by $H_R:=R(X)$ the range of the operator $R$ and we recall that $(\ker R)^{\bot}=\overline{R(X)}$. In order to provide $H_R$ with a Hilbert structure, we recall that the restriction $R_{|_{(\ker R)^{\bot}}}$ is a injective operator, and so
\[
R_{|_{(\ker R)^{\bot}}}:(\ker R)^{\bot}\subseteq X\rightarrow H_R
\] 
is bijective. We call pseudo-inverse of $R$ the liner bounded operator $R^{-1}:H_R\rightarrow X$ where for all $y\in H_R$ $R^{-1}y$ is the unique $x\in (\ker R)^{\bot}$ such that $Rx=y$, see \cite[Appendix C]{LI-RO1}. We introduce the scalar product 
\begin{equation}\label{Rprod}
\scal{x}{y}_{H_R}:=\langle R^{-1}x,R^{-1}y\rangle_X,\quad x,y\in H_R
\end{equation}
and its associated norm $\norm{x}_{H_R}:=\|R^{-1}x\|_X$. With this inner product $H_R$ is a separable Hilbert space and a Borel subset of $X$ (see \cite[Theorem 15.1]{KE1}). A possible orthonormal basis of $H_R$ is given by $\{Re_k\}_{k\in\Nat}$, where $\{e_k\}_{k\in\Nat}$ is any orthonormal basis of $(\ker R)^{\bot}$. Denoting by $P$ the orthogonal projection on $\ker R$, we recall that 
\begin{align}
RR^{-1} &=\Id_{H_R},\qquad R^{-1}R=\Id_{X}-P.\label{orietta1}
\end{align}
Notice that for every $x\in H_R$
\[
\norm{x}_X=\|RR^{-1}x\|_X\leq \norm{R}_{\mathcal{L}(X)}\|R^{-1}x\|_{X}\leq \norm{R}_{\mathcal{L}(X)}\norm{x}_{H_R}.
\]

The following notion of differentiability first appeared in \cite{GRO1} and \cite{KUO1}.

\begin{dfn}\label{intr_defn_diff} 
We say that a function $\ph:X\rightarrow \R$ is $H_R$-differentiable at $x\in X$ if there exists  $L_x\in\mathcal{L}(H_R;\R)$ such that 
\begin{align*}
\lim_{\norm{h}_{H_R}\rightarrow 0}\frac{|\ph(x+h)-\ph(x)-L_xh|}{\norm{h}_{H_R}}=0.
\end{align*}
In this case $L_x$ is unique and we set $D_{H_R}\ph(x):=L_x$. We say that $\ph$ is $H_R$-differentiable if $\ph$ is $H_R$-differentiable at every $x\in X$.
Since $H_R$ is a Hilbert space, by the Riesz representation theorem for every $x\in X$ there exists a unique $l_x\in H_R$ such that
\[
D_{H_R} \ph(x)h=\langle l_x,h\rangle_{H_R},\qquad h\in H_R.
\]
We call $l_x$ the $H_R$-gradient of $\ph$ at $x\in X$ and we denote it by $\nabla_{H_R} \ph(x)$. We denote by $C^1_{b,H_R}(X)$ the subspace of $C_b(X)$ of the $H_R$-differentiable functions $\ph:X\rightarrow\R$ such that $\nabla_{H_R}\ph\in C_b(X;H_R)$.
\end{dfn}

\begin{prop}\label{dalpha}
If $\ph\in C^1_b(X)$, then $\ph\in C^1_{b,H_R}(X)$. Moreover for every $x\in X$ $\nabla_{H_R}\ph(x)=R^2\nabla \ph(x)$ and $\norm{\nabla_{H_R} \ph (x)}_{H_R}=\norm{R\nabla \ph (x)}_{X}$. 
\end{prop}
\begin{proof}
Let $\ph\in C^1_b(X)$, $x\in X$ and $h\in H_R\sm\{0\}$. We have
\begin{align}
\abs{\frac{\ph (x+h)-\ph (x)-\ix{\nabla \ph (x)}{h}}{\norm{h}_{H_R}}}=\abs{\frac{\ph (x+h)-\ph (x)-\ix{\nabla \ph (x)}{h}}{\norm{h}_X}}\frac{\norm{h}_X}{\norm{h}_{H_R}}.
\end{align} 
Since $H_R$ is continuously embedded in $X$, $\displaystyle{\frac{\norm{h}_X}{\norm{h}_{H_R}}}$ is bounded and $\ph $ is $H_R$-differentiable at $x\in X$. 
Moreover $\ix{\nabla \ph (x)}{h}=\iprod{\nabla_{H_R} \ph (x)}{h}_{H_R}$ for $h\in H_R$. Let $P$ be the orthonormal projection on $\ker R$, we get 
\begin{align}
\iprod{\nabla_{H_R} \ph (x)}{h}_{H_R}&=\ix{\nabla \ph (x)}{h}=\ix{\nabla \ph (x)}{(I-P)h}=\ix{\nabla \ph (x)}{R^{-2}R^2h}\nonumber \\
&=\iprod{R^2\nabla \ph (x)}{h}_{H_R},\nonumber \\
\norm{\nabla_{H_R} \ph (x)}_{H_R}^2&=\iprod{R^2\nabla \ph (x)}{R^2\nabla \ph (x)}_{H_R}=\ix{R^{-1}R^2\nabla \ph (x)}{R^{-1}R^2\nabla \ph (x)}\nonumber\\
&=\ix{(I-P)R\,\nabla \ph (x)}{(I-P)R\,\nabla \ph (x)}=\norm{R\,\nabla \ph (x)}_{X}^2. \nonumber
\end{align} 

Since $\nabla \ph (x)$ is continuous at $x$ for all $x\in X$, we obtain that $\ph$ belongs to $C^1_{b,H_R}(X)$.
\end{proof}

%By Proposition \ref{dalpha} follows immediately that
%\[
%C^1_b(X)\subseteq  C^1_{b,H_R}(X).
%\]
\section{The evolution operator \texorpdfstring{$\pst$}~~and gradient estimates}\label{pstegradiente}

Let $(X, \ix{\cdot}{\cdot},\norm{\,\cdot\,}_X)$ be a separable Hilbert space. In this section we define a class of evolution operators acting on $B_b(X)$, which is the non autonomous version of Mehler semigroups.  

Let us state the basic hypothesis of our framework. Let $\Delta=\{(s,t)\in\R^2\ \mbox{s.t.}\ s<t\}$.
\begin{ip}\label{1} 
\leavevmode
\begin{enumerate} 
\item $\{U(t,s)\}_{(s,t)\in\overline{\Delta}}\subseteq\op(X)$ is a strongly continuous evolution operator, namely for every $x\in X$ the map
\begin{equation}
(s,t)\in\overline{\Delta}\longmapsto U(t,s)x\in X,
\end{equation}
is continuous and
\begin{enumerate}
\item $U(t,t)=I$ for every $t\in\R$,
\item $U(t,r)U(r,s)=U(t,s)$ for $s\leq r\leq t$.
\end{enumerate}
Moreover we assume that there exist $M>0$ and $\zeta\in\R$ such that 
\begin{equation}\label{omega}
\norm{U(t,s)}_{\op(X)}\leq Me^{-\zeta(t-s)}.
\end{equation}
\item $\{B(t)\}_{t\in\R}\subseteq\op(X)$ is a bounded family of strongly continuous linear and bounded operators, namely
\begin{enumerate}
\item there exists $K>0$ such that 
\begin{equation}
\sup_{t\in\R}\norm{B(t)}_{\op(X)}\leq K,
\end{equation}
\item the map 
\begin{equation}
t\in\R\mapsto B(t)x\in X
\end{equation}
is continuous for every $x\in X$.
\end{enumerate}
\item The map $f:\R\longrightarrow X$ is bounded and measurable.
\item For every $(s,t)\in\Delta$ the operator $Q(t,s):X\rightarrow X$ given by
\begin{align} \label{qtscov}
Q(t,s)&=\int_s^tU(t,r)B(r)B(r)^\star U(t,r)^\star \,dr,
\end{align}
has finite trace.
\end{enumerate}
\end{ip}
\noindent In this paper we will study the evolution operator $\{P_{s,t}\}_{(s,t)\in\overline{\Delta}}$ defined by
\begin{align}
P_{r,r}&=I,\qquad \forall\; r\in\R, \\
P_{s,t}\ph(x)&=\int_X\ph(y)\,\Ng_{U(t,s)x ,Q(t,s)}(dy), \ \ (s,t)\in\Delta,\ \ph\in B_b(X),
\end{align}
where $\Ng_{U(t,s)x ,Q(t,s)}$ is the Gaussian measure on $\mathcal{B}(X)$ with mean $U(t,s)x$
and covariance operator $Q(t,s)$ given by \eqref{qtscov}.

\begin{oss}
We emphasize that, by the Fernique Theorem, it is possible to define $\pst$ on Borel measurable functions with power growth, namely for Borel measurable functions $\ph:X\longrightarrow\R$ such that there exists $C,m>0$ such  that 
\begin{equation}\label{limsup}
\abs{\ph(x)}\leq C(1+\norm{x}_X^m),\quad x\in X.
\end{equation}
Moreover $\pst$ leaves invariant the space of Borel measurable functions having fixed power growth $m>0$. Indeed, by \cite[Thm. 2.6]{cerlun} $\pst\ph$ is a  Borel measurable. Moreover, if $\ph$ satisfies \eqref{limsup}, then for every $x\in X$ we have
\begin{align*}
\abs{\pst\ph(x)}&\leq \int_X \abs{\ph(y+\uts x)}\,N_{0,Q(t,s)}(dy)\leq C \int_X \left[1+\left(\norm{y}_X+\norm{\uts x}_X\right)^m\right]\,N_{0,Q(t,s)}(dy)\nonumber\\
&\leq C_m\left[1+ \norm{U(t,s)}^m_{\mathcal{L}(X)}\norm{x}_X^m+ \int_X\norm{y}_X^{m}\,\Ng_{0,Q(t,s)}(dy)\right], 
\end{align*}
where $C_m$ is a positive constant.
\end{oss}

We conclude this section studying some regularization properties of $\pst$. In the autonomous case the smoothing properties of Ornstein-Uhlenbeck semigroups are well known, see for instance \cite{bigna22,bigna22col,MR2299922,MR4011050,MR4311102,maspri,MR1976297,MR1985790,lun-pal2020}. Time dependency of diffusion operator ($Q(t):=B(t)B(t)^\star$) yelds significant differences in the regularity properties of $\pst$. In order to study such properties, for any $r\in\R$ we define the space
\[
\ac_r:=H_{\Qp{r}}:=\Qp{r}(X).
\]
We refer to Subsection \ref{HR} for a description of this space.

Let $E$ be a subspace of $X$. In the following we often denote $U(t,s)_{|_E}$ by $U(t,s)$ by abuse of language.

\begin{prop}\label{regck} 
Assume that Hypothesis \ref{1} holds true and that for every $(s,t)\in\overline{\Delta}$
\[
\uts\in\op(\ac_s;\ac_t).
\]
 %that there  exists a map $\be:\Delta\subseteq \R^2\longrightarrow(0,+\infty)$ such that for every $(s,t)\in\Delta$
%\[
%\norm{U(t,s)}_{\op(\ac_s;\ac_t)}\leq \be(s,t).
%\] 
Then for every $(s,t)\in\overline{\Delta}$ we have
\begin{align*}
P_{s,t}(C^1_{b,\ac_t}(X))\subseteq C^1_{b,\ac_s}(X),
\end{align*}
(see Definition \ref{intr_defn_diff}).
Moreover for every $\varphi\in C^1_{b,H_t}(X)$, $x\in X$ and $h\in H$ 
\begin{align}
&D_{\ac_s}(\pst\ph)(x)h=\pst \bigl(D_{\ac_t}\ph(\cdot)\uts h\bigr)(x),\label{dk}\\
&\norm{\nabla_{\ac_s}\pst\ph(x)}_{\ac_s}\leq \norm{\uts}_{\op(\ac_s;\ac_t)}\pst\left(\norm{\nabla_{\ac_t}\ph(\cdot)}_{\ac_t}\right)(x)\label{stimachemiserve1}. 
\end{align}
\end{prop}

\begin{proof} \eqref{dk} is proven in \cite{defazio22} and \eqref{stimachemiserve} is a straightforward consequence of \eqref{dk}.
\end{proof}

\begin{oss}
In view of Proposition \ref{dalpha} for every $\varphi\in C^1_{b}(X)$, $x\in X$ and $(s,t)\in\Delta$, inequality \eqref{stimachemiserve1} reads as
\begin{equation}\label{stimachemiserve}
\norm{Q(s)^{\frac{1}{2}}\nabla\pst\ph(x)}_{X}\leq \norm{\uts}_{\op(\ac_s;\ac_t)}\pst\left(\norm{Q(t)^{\frac{1}{2}}\nabla\ph(\cdot)}_{X}\right)(x). 
\end{equation}
\end{oss}

\section{Connections between \texorpdfstring{$\{\pst\}_{(s,t)\in\Delta}$}~ and \texorpdfstring{$\{L(r)\}_{r\in\R}$}~}\label{connessione}
One of the main issues working on non autonomous problems is the lack of similar theories to the ones of strongly continuous or analytic semigroups for evolution operators. We cannot even define the weak genertor of $\pst$ via Laplace transform as in the case of Ornstein-Uhlenbeck semigroups in \cite{MR1293091,MR4553575,MR1724248}.
In this section we prove that for suitable functions $\varphi:X\rightarrow \R$ we have
\begin{align}
&\frac{\de}{\de s} \pst \ph (x)=-L(s)\pst \ph (x),\quad (s,t)\in\Delta,\; x\in X,\label{deri1}\\
&\frac{\de}{\de t} \pst \ph (x)=\pst L(t) \ph (x),\quad (s,t)\in\Delta,\; x\in X,\label{deri2}
\end{align}
where $\{L(r)\}_{r\in\R}$ is the family of operators given by
\begin{equation}\label{OU}
L(r)\ph(x)=\frac{1}{2}\mbox{\rm Tr}\Bigl(Q(r)D^2\ph(x)\Bigr)+\ix{x}{A(r)^\star\nabla\ph(x)},
\end{equation} 
and $\{A(r)\}_{t\in\R}$ is a family of linear and not necessarily bounded operators associated to $\{U(t,s)\}_{(s,t)\in\overline{\Delta}}$ in the following way.

\begin{ip}\label{2bis}
Assume that Hypothesis \ref{1} holds true and that in addition there exists a family of linear operators $A(r):D(A(r))\subseteq X\rightarrow X$, $r\in\R$, satisfying
\begin{enumerate}
\item[(i)] $D(A(r))$ and $D(A(r)^\star)$ are dense in $X$ for every $r\in\R$. 
\item[(ii)] For every $(s,t)\in\Delta$ we have
\begin{align*}
&\uts D(A(s))\subseteq D(A(t)),\\
&\uts^\star D(A(t)^\star)\subseteq D(A(s)^\star).
\end{align*}
\item[(iii)] For every $(s,t)\in\Delta$ and  $x\in D(A(s))$ we have
\begin{align}
&\frac{\de}{\de s}U(t,s)x=-U(t,s)A(s)x,\label{deri1A}\\
&\frac{\de}{\de t}U(t,s)x=A(t)U(t,s)x.\label{deri2A}
\end{align}
\end{enumerate}
\end{ip}

\begin{oss} In Section \ref{Examples} we show that if the  family $\{A(r)\}_{r\in\R}$ is associated to a non autonomous abstract parabolic problem in the sense of Acquistapace-Terreni (\cite{MR945820,MR934508}) then $\{A(r)\}_{r\in\R}$ satisfies Hypothesis \ref{2bis}. 
\end{oss}

\begin{oss}
By Hypotheses \ref{2bis} it follows immediately that for every $(s,t)\in\Delta$ and $x\in D(A(t)^\star)$ we have 
\begin{align}
&\frac{\de}{\de s}U(t,s)^\star x=-A(s)^\star U(t,s)^\star x,\label{deri1A*}\\
&\frac{\de}{\de t}U(t,s)^\star x=U(t,s)^\star A(t)^\star x.\label{deri2A*}
\end{align}
\end{oss}

\subsection{Cylindrical functions}\label{exp}

In this subsection we define a space of smooth cylindrical functions such that \eqref{deri1} and \eqref{deri2} hold true. We define suitable trigonometric polynomials on $X$ and we introduce the space of Bohr almost periodic functions that will be crucial in Section \ref{log-sobolev}.

Throughout this subsection we fix $r\in\R$. 

\begin{dfn}[Trigonometric polynomials]\label{trigometrici}
Let $V$ be a subspace of $X$. We denote by $\mathcal{E}(X;V)$ the linear span of all real and imaginary parts of the functions 
\begin{equation}\label{phih}
x\rightarrow \varphi^{(h)}(x):=e^{i\iprod{x}{h}_X},
\end{equation}
 where $h\in V$ (we shall omit $h$ from the notation $\varphi^{(h)}$ when it is not necessary). 
\end{dfn}

\begin{oss}
We note that ${\rm Trig}(\R^n):=\mathcal{E}(\R^n;\R^n)$ is the usual space of trigonometric polynomials on $\R^n$.
\end{oss}

We set
\begin{equation}\label{trigonometriciR}
\mathcal{E}_r(X):=\mathcal{E}(X;D(A(r)^\star)).
\end{equation}

\begin{oss}\label{funzioniE}The set of functions $\mathcal{E}_r(X)$ is often used in the autonomous case in which it represents a core for Ornstein-Uhlenbeck type operators in $L^p$ spaces with respect to the invariant measure, see for instance \cite{MR1985790}. However in general there is not a dense subspace  $\displaystyle{\mathfrak{D}\subseteq\bigcap_{r\in\R} D(A(r)^\star)}$, which prevents from using a unique space independent of $r$.
\end{oss}

Let $L(r)$ be the operator defined in \eqref{OU}. If $h\in D(A(r)^\star)$ and $\varphi^{(h)}$ is defined by  \eqref{phih}, then we have
\begin{equation}\label{xiAgeneratore}
L(r)\varphi^{(h)}(x)=\left[i\iprod{x}{A(r)^\star h}_X-\frac{1}{2}\norm{\Qp{r}h}^2_X\right]\varphi^{(h)}(x),\quad x\in X.
\end{equation}

Here we introduce a space of functions that contains $\mathcal{E}_{r}(X)$ and that will be used in the proofs of Section \ref{log-sobolev}. 
\begin{dfn}\label{cilindriche}
Let $k\in\Nat\cup\{0\}$. We denote by $\mathcal{F}_{r}C^k_b(X)$ the space of functions $\ph$ such that there exists $n\in\Nat$, $\psi\in C_b^k(\R^n)$ and $h_1,\ldots h_n\in D(A(r)^\star)$ orthonormal such that 
\begin{equation}\label{tcilindrico}
\ph(x)=\psi\left(\iprod{x}{h_1},\ldots,\iprod{x}{h_n}\right),\qquad x\in X.
\end{equation} 
If $k=0$ we write $\mathcal{F}_{r}C_b(X)$ instead of $\mathcal{F}_{r}C^0_b(X)$.
\end{dfn}

%\begin{oss} Definition \ref{cilindriche} is equivalent to the standard definition of cylindrical functions given in \cite{MR1642391}. Indeed, if  $$\ph(x)=\psi\left(\iprod{x}{h_1},\ldots,\iprod{x}{h_n}\right),\qquad x\in X,$$ with $h_1,...,h_n$ linearly dependent, there exists $j\leq n$  such that $h_1,...,h_j$ linearly independent. By the Gram-Schmidt process we can choose $\widetilde{h}_1,...,\widetilde{h}_j\in D(A(r)^\star$ orthonormal and there exists $\widetilde{\psi}\in C_b^k(\R^j)$ such that $$\ph(x)=\widetilde{\psi}\left(\iprod{x}{\widetilde{h}_1},\ldots,\iprod{x}{\widetilde{h}_j}\right),\qquad x\in X.$$
%\end{oss}
%

Let $\varphi\in\mathcal{F}_{r}C^2_b(X)$ be given by \eqref{tcilindrico}. Since  $h_1,\ldots h_n\in D(A(r)^\star )$ are orthonormal then %we can extend them to an orthonormal basis $\{h_k\}_{k\in\Nat}$ of $X$ (see \cite[Thms 1.54 and 1.55]{fabian_banach_2011}). Hence 
\begin{align}\label{ftgeneratore}
L(r)\ph(x)&=\frac{1}{2}\sum_{i=1}^n \ix{Q(r)\nabla^2\varphi(x)h_i}{h_i}+\sum_{i=1}^n\ix{x}{A(r)^\star h_i}\ix{\nabla\ph(x)}{h_i},\quad x\in X,
\end{align}
where $L(r)$ is the operator defined in \eqref{OU}.
By \eqref{trigonometriciR} and Definitions \ref{trigometrici} and \ref{cilindriche} it follows immediately that
\begin{equation}\label{EFC}
\mathcal{E}_r(X)\subseteq \mathcal{F}_{r}C^k_b(X),\quad \forall\; k\in\Nat.
\end{equation}

\begin{dfn}\label{almostp}
Let $n\in\Nat$ and $\ph\in C_b(\R^n)$. We say that $\ph$ is Bohr almost periodic if for every $\ep>0$ there exists $\rho>0$ such that for all $x_0\in \R^n$ there exists $\tau\in B(x_0,\rho)$ such that 
\begin{equation}
\abs{\ph(x+\tau)-\ph(x)}<\ep,\qquad \forall\;x\in \R^n.
\end{equation}
We denote by $AP_b(\R^n)$ the subspace of $C_b(\R^n)$ of Bohr almost periodic functions from $\R^n$ to $\R$.

\noindent Moreover we denote by $AP^2_b(\R^n)$ the subspace of $C^2_b(\R^n)\cap AP_b(\R^n)$ of the functions $\ph:\R^n\longrightarrow\R$ such that the partial derivatives $\displaystyle{\frac{\de \ph}{\de x_i}}$ and  $\displaystyle{\frac{\de^2 \ph}{\de x_i\de x_j}}$ belongs to $AP_b(\R^n)$, for every $i,j\in\Nat$. 
\end{dfn}

For more details about Bohr almost periodic functions in several variables we refer to \cite{MR1120781,MR4552385,MR4510388}.

\begin{dfn}\label{almostpX}
We denote by $\mathfrak{B}^2_r(X)$ the subspace of $\mathcal{F}_rC_{b}^2(X)$ of the functions $\varphi:X\rightarrow\R$ given by \eqref{tcilindrico} with $\psi\in AP^2_b(\R^n)$, for some $n\in\Nat$. 
%If $k=0$ we write $\mathfrak{B}_r(X)$ instead of $\mathfrak{B}^0_r(X)$.
\end{dfn}

\begin{prop}\label{approssimazionechit}$ $
Let $h\in\Nat\cup \{0\}$. For every $\varphi\in \mathcal{F}_rC^h_b(X)$ there exist a sequence $\{\varphi_k\}_{k\in\Nat}\subseteq\mathcal{E}_r(X)$ and $C>0$ such that 
\begin{align}\label{dom0}
\norm{\varphi_{k}}_{C^h_b(X)}\leq C\norm{\varphi}_{C^h_b(X)},\qquad k\in\Nat,
\end{align}
and for every $x\in X$ we have
\begin{align}\label{conv0}
\lim_{k\rightarrow +\infty}\Big(\vert \varphi_{k}(x)-\varphi(x)\vert +\sum^h_{j=1}\| \nabla^j\varphi_{k}(x)-\nabla^j\varphi(x)\|_{\mathcal{L}^{j-1}(X)}\Big)=0.
\end{align}
Moreover if $\varphi\in \mathfrak{B}^2_r(X)$ then
\begin{equation}\label{convergenza-grafico}
\lim_{k\rightarrow +\infty}\left(\norm{\varphi_k-\varphi}_{C^2_b(X)}+\sup_{x\in X}\frac{\norm{L(r)\varphi_k-L(r)\varphi}_{X}}{1+\norm{x}_X}\right)=0.
\end{equation}

\end{prop}
\begin{proof}
Let $h\in\Nat\cup\{0\}$ and let $\varphi\in \mathcal{F}_rC^h_b(X)$. There exist $n\in\Nat$, $\psi\in C^h_b(\R^n)$ and $h_1,\ldots h_n\in D(A(r)^\star )$ orthonormal such that 
\begin{equation}\label{cilC}
\ph(x)=\psi\left(\iprod{x}{h_1},\ldots,\iprod{x}{h_n}\right),\qquad x\in X.
\end{equation}
We define the orthogonal projection $P_n^r$ on ${\rm span}\{h_1,\ldots,h_n\}$, namely
\begin{equation}\label{proR}
P_n^rx:=\sum_{k=1}^n\scal{x}{h_k}h_k.
\end{equation}
We denote by $\mathfrak{I}_n^r:P_n^r(X)\rightarrow\R^n$ the canonical isometry given by
\begin{equation}\label{isoR}
\mathfrak{I}_n^r h_j=e_j,\qquad j\in\{1,\ldots,n\},
\end{equation}
where $e_1,\ldots e_n$ is the canonical basis of $\R^n$. By the definitions of $P_n^r$ and $\mathfrak{I}_n^r$ formula \eqref{cilC} reads as
\[
\ph(x)=\psi\left(\mathfrak{I}_n^r P_n^rx\right),\qquad x\in X.
\]
By \cite[Lemma 8.1]{DA-LU-TU1}, there exists $C>0$ (depending only on $n$) and a sequence of trigonometric polynomials $\{\psi_k\}_{k\in\Nat}\subseteq  {\rm Trig}(\R^n)$ such that
\begin{align*}
&\norm{\psi_{k}}_{C^h_b(\R^n)}\leq C\norm{\varphi}_{C^h_b(\R^n)},\qquad\qquad\qquad\qquad\,\, k\in\Nat,\\
\lim_{k\rightarrow +\infty}&\Big(\vert \psi_{k}(x)-\psi(x)\vert +\sum_{j=1}^h\| \nabla^j\psi_{k}(x)-\nabla^j\psi(x)\|_{\mathcal{L}^{j-1}(\R^n)}\Big)=0,\quad x\in\R^n.
\end{align*}
Setting $\varphi_k(\cdot):=\psi_k(\mathfrak{I}_n^rP^r_n(\cdot))$ for every $k\in\Nat$ and recalling that $\|P_n^r\|_{\mathcal{L}(X)}=\|\mathfrak{I}_n^r\|_{\mathcal{L}(X)}=1$ we obtain \eqref{dom0} and \eqref{conv0}. Now we prove that $\psi_{k}\in {\rm Trig}(\R^n)$ implies $\ph_k\in  \mathcal{E}_r(X)$, for every $k\in\Nat$. Indeed, for every $h\in \R^n$ and setting $$\Phi^{(h)}(\xi)=e^{\iprod{h}{\xi}_{\R^n}},\ \ \xi\in\R^n,$$  we get
\begin{align}\label{prova}
\Phi^{(h)}(\mathfrak{I}_n^rP_n^r x)=e^{i\ix{\mathfrak{I}_n^rP_n^r(x)}{h}}=e^{i\ix{x}{{\mathfrak{I}_n^r}^\star h}},\qquad x\in X.
\end{align}
Since ${\mathfrak{I}_n^r}^\star h\in P_n^r(X)$ and $h_1,\ldots h_n\in D(A(r)^\star)$ then ${\mathfrak{I}_n^r}^\star h\in D(A(r)^\star )$ and so $\ph_k\in \mathcal{E}_r(X)$, for every $k\in\Nat$ (see Definition \ref{trigometrici} and \eqref{trigonometriciR}).

Now we prove \ref{convergenza-grafico}. Let  $\varphi\in \mathfrak{B}^2_r(X)$ given by $$\ph(x)=\psi\left(\ix{x}{h_1},...,\ix{x}{h_n}\right),\ \ x\in X,$$
where $\psi\in AP^2(\R^n)$ and $h_1,...,h_n\in D(A(r)^\star)$. By \cite[Prop.\,6.1]{MR1120781} $AP^2_b(\R^n)$ is the closure in $C^2_b(\R^n)$ of ${\rm Trig}(\R^n)$, then there exists $\{\psi_k\}_{k\in\Nat}\subseteq{\rm Trig}(\R^n)$ such that
\[
\lim_{k\rightarrow +\infty}\norm{\psi_k-\psi}_{C_b^2(\R^n)}=0.
\] 
Hence we get
\begin{equation}\label{step1convergenzagrafico}
\lim_{k\rightarrow +\infty}\norm{\varphi_k-\varphi}_{C_b^2(X)}=0,
\end{equation}
where the functions $\ph_k$ are defined as in the previous approximation procedure.  
Finally combining \eqref{ftgeneratore} and \eqref{step1convergenzagrafico} we obtain \eqref{convergenza-grafico}.
\end{proof}

It is well know that trigonometric polynomials are not dense in $C_b(X)$ even if $X=\R$, however it is possible to prove the following weaker approximation result.

\begin{prop}\label{approssimazione}
Let $h\in\Nat\cup\{0\}$. For every $\varphi\in C^h_b(X)$ there exist a 2-sequence $\{\varphi_{n,m}\}_{n,m\in\Nat}\subseteq \mathcal{E}_{r}(X)$, $\{\varphi_{n}\}_{n\in\Nat}\subseteq \mathcal{F}_{r}C_b^h(X)$ and $\{c_n\}_{n\in\Nat}\subseteq [0,+\infty)$ such that 
\begin{align}
&\norm{\varphi_{n,m}}_{C^h_b(X)}\leq c_n\norm{\varphi_n}_{C^h_b(X)},\qquad\qquad\qquad\qquad\quad\phantom{aaaa} n, m\in\Nat,\label{1c}\\
&\norm{\varphi_{n}}_{C^h_b(X)}\leq\norm{\varphi}_{C^h_b(X)},\qquad\qquad\qquad\qquad\quad\qquad\qquad \ \ n\in\Nat,\label{2c}\\
\lim_{m\rightarrow +\infty}\Big(\vert \varphi_{n,m}(x)&-\varphi_n(x)\vert +\sum_{j=1}^h\| \nabla^j\varphi_{n,m}(x)-\nabla^j\varphi_n(x)\|_{\mathcal{L}^{j-1}(X)}\Big)=0,\quad x\in X, \ n\in\Nat,\label{3c}\\
\lim_{n\rightarrow +\infty}\Big(\vert \varphi_{n}(x)-&\varphi(x)\vert +\sum_{j=1}^h\| \nabla^j\varphi_{n}(x)-\nabla^j\varphi(x)\|_{\mathcal{L}^{j-1}(X)}\Big)=0,\qquad\phantom{aaaaa} x\in X.\label{4c}
\end{align}
\end{prop}
\begin{proof}
Since $D(A(r)^\star )$ is dense in $X$ there exists an orthonormal basis $\{e^r_k\}_{k\in\Nat}$ of $X$ such that $e^r_k\in D(A(r)^\star )$, for every $k\in\Nat$. 
Let $n\in\Nat$ and let $P_n^r$ be the orthogonal projection on ${\rm span}\{e^r_1,\ldots e^r_n\}$.  Let $h\in\Nat\cup\{0\}$ and let $\varphi\in C^h_b(X)$. We define
\[
\varphi_{n}(x):=\varphi(P^r_{n}x),\quad x\in X.
\]
Since $\varphi\in C^h_b(X)$ and $\norm{P_n^r}_{\mathcal{L}(X)}=1$, we obtain \eqref{2c} and \eqref{4c}.
Fixed $n\in\Nat$ the function $\varphi_n$ belongs to $\mathcal{F}_rC_b(X)$ so by Proposition \ref{approssimazionechit} there exists a sequence $\{\varphi_{n,m}\}_{m\in\Nat}\subseteq \mathcal{E}_r(X)$ and $c_n>0$ such that \eqref{1c} and \eqref{3c} are verified.
\end{proof}

\begin{oss}\label{densità} Let $\g$ be a Borel probability measure on $X$. By Proposition \ref{approssimazione} $\mathcal{E}_r(X)$ is dense in $L^p(X,\g)$ for all $p\geq 1$.
\end{oss}

\subsection{Differentiation formulas for \texorpdfstring{$\pst$}~~}\label{Hyperframe}

In this subsection we prove formulas \eqref{deri1} and \eqref{deri2}. 
\begin{lem}\label{lemma47}
Assume that Hypothesis \ref{2bis} holds true. Fix $(s,t)\in\Delta$, $h\in X^t$ and let $\ph^{(h)}$ be defined by \eqref{phih}. Then
{\small \begin{align}
&\pst\varphi^{(h)}(x)=e^{-\frac{1}{2}\ix{Q(t,s) h}{h}}\varphi^{(U(t,s)^\star h)}(x).\label{st}
\end{align}
It follows that $\pst(\mathcal{E}_t(X))\subseteq \mathcal{E}_s(X)$. Furthermore for all $x\in X$ we have
\begin{align}
L(s)\pst \ph^{(h)} (x)&=\biggl[i\iprod{x}{A(s)^\star \uts^\star h}_X-\frac{1}{2}\norm{\Qp{s}\uts^\star h}^2_X\biggl] \pst\ph^{(h)}(x), \label{Lfuori}\\
\pst L(t)\ph^{(h)}(x)&=\biggl[i\ix{x}{\uts^\star A(t)^\star h}-\ix{Q(t,s)A(t)^\star h}{h}-\frac{1}{2}\norm{\Qp{t}h}^2_X\biggr]\pst\ph^{(h)}(x).\label{Ldentro}
\end{align}}
\end{lem}
\begin{proof}
Let $h\in X^t$ and let $\ph^{(h)}$ be defined by \eqref{phih}. For all $x\in X$ we get
\begin{align}
\pst \ph^{(h)}(x)&=\int_X e^{i\ix{h}{y}}\,\Ng_{U(t,s)x , Q(t,s)}(dy)=\rwhat{\Ng}_{U(t,s)x , Q(t,s)}(h) \nonumber\\
&=e^{i\ix{h}{U(t,s)x }}e^{-\frac{1}{2}\ix{Q(t,s) h}{h}}=e^{-\frac{1}{2}\ix{Q(t,s) h}{h}}\varphi^{(U(t,s)^\star h)}(x), \nonumber
\end{align}
so that \eqref{st} holds and $\pst$ maps $\mathcal{E}_t(X)$ into $\mathcal{E}_s(X)$ for every $(s,t)\in\Delta$.

Recalling that by Hypothesis \ref{2bis} $\uts^\star h\in X^s$, we have $\pst(\mathcal{E}_t(X))\subseteq \mathcal{E}_s(X)$. Let us prove \eqref{Lfuori}. By \eqref{xiAgeneratore} for all $x\in X$  we have
\begin{align}
L(s)\pst \ph^{(h)} (x)&= L(s)\left(e^{-\frac{1}{2}\ix{Q(t,s)h}{h}}\ph^{(\uts^\star h)}\right)(x)\nonumber\\
&=e^{-\frac{1}{2}\ix{Q(t,s)h}{h}}\biggl[i\iprod{x}{A(s)^\star \uts^\star h}_X-\frac{1}{2}\norm{\Qp{s}\uts^\star h}^2_X\biggl]\varphi^{(\uts^\star h)}(x)\nonumber \\
&=\biggl[i\iprod{x}{A(s)^\star \uts^\star h}_X-\frac{1}{2}\norm{\Qp{s}\uts^\star h}^2_X\biggl] \pst\ph^{(h)}(x)\nonumber
\end{align}
and \eqref{Lfuori} holds. To prove \eqref{Ldentro}, we combine \eqref{xiAgeneratore} and \eqref{st} and for all $x\in X$ we get
{\small\begin{align}
\pst L(t)\ph^{(h)}(x)&=\pst\biggl(\left[i\iprod{\cdot}{A(t)^\star h}_X-\frac{1}{2}\norm{\Qp{t}h}^2_X\right]\varphi^{(h)}(\cdot)\biggr)(x)\nonumber \\
&=\int_X\left(i\ix{U(t,s)x +y}{A(t)^\star h}\varphi^{(h)}\left(U(t,s)x +y\right)\right)\,\Ng_{0,Q(t,s)}(dy)-\frac{1}{2}\norm{\Qp{t}h}^2_X\pst\varphi^{(h)}(x)\nonumber \\
&=: I_1-\frac{1}{2}\norm{\Qp{t}h}^2_X\pst\varphi^{(h)}(x).\nonumber
\end{align}}
We note that 
\begin{align}
I_1&=\int_X\left(i\ix{U(t,s)x +y}{A(t)^\star h}\varphi_h\left(U(t,s)x +y\right)\right)\,\Ng_{0,Q(t,s)}(dy)\nonumber\\
&=\int_X\frac{\de}{\de\left(A(t)^\star h\right)}e^{i\ix{U(t,s)x +y}{h}}\,\Ng_{0,Q(t,s)}(dy)=\frac{\de}{\de\left(A(t)^\star h\right)}\int_Xe^{i\ix{y}{h}}\,\Ng_{U(t,s)x ,Q(t,s)}(dy)\nonumber \\
&=\frac{\de}{\de\left(A(t)^\star h\right)}\widehat{\Ng}_{U(t,s)x ,Q(t,s)}(h)=\frac{\de}{\de(A(t)^\star h)}\left(e^{-\frac{1}{2}\ix{Q(t,s)h}{h}}e^{i\ix{x}{\uts^\star h}}\right)\nonumber\\
&=\left[i\ix{x}{\uts^\star A(t)^\star h}-\ix{Q(t,s)A(t)^\star h}{h}\right]e^{-\frac{1}{2}\ix{Q(t,s)h}{h}}e^{i\ix{x}{\uts^\star h}}\nonumber\\
&=\left[i\ix{x}{\uts^\star A(t)^\star h}-\ix{Q(t,s)A(t)^\star h}{h}\right]\pst\ph^{(h)}(x).\nonumber
\end{align}
Summing up, \eqref{Ldentro} follows.
%\begin{align}
%\pst L(t)\ph^{(h)}(x)&=\biggl[i\ix{x}{\uts^\star A(t)^\star h}-\ix{Q(t,s)A(t)^\star h}{h}-\frac{1}{2}\norm{\Qp{t}h}^2_X\biggr]\pst\ph^{(h)}(x).\nonumber
%\end{align}
\end{proof}

\begin{lem}
Assume that Hypothesis \ref{2bis} holds true. For each $s_0,t_0\in\R$, we have
{\small\begin{align}\label{derivataqts}
&\left(\frac{d}{dt}\ix{Q(t,s)h}{h}\right)\at{t=t_0}=\norm{\Qp{t_0}h}_X^{\frac{1}{2}}+2\ix{Q(t_0,s)A(t_0)^\star h}{h},\ \ h\in D(A(t_0)^\star),\ s\leq t, t_0,\\
&\left(\frac{d}{ds}\ix{Q(t,s)x}{x}\right)\at{s=s_0}=-\ix{U(t,s_0) Q(s_0)U(t,s_0)^\star x}{x},\ \ x\in X,\ t\geq s, s_0. \label{derivataqts2}
\end{align}}
\end{lem}
\begin{proof}
Let $h\in D(A(t_0)^\star)$ and $\ep>0$. Then for $s\leq t_0$ we have
\begin{align}
&\frac{\ix{Q(t_0+\ep,s)h}{h}-\ix{Q(t_0,s)h}{h}}{\ep}\nonumber \\
&=\frac{1}{\ep}\left(\int_s^{t_0+\ep}\ix{U(t_0+\ep,r)Q(r)U(t_0+\ep,r)^\star h}{h}\,dr-\int_s^{t_0}\ix{U(t_0,r)Q(r)U(t_0,r)^\star h}{h}\,dr\right)\nonumber \\
&=\frac{1}{\ep}\biggl[\int_s^{t_0}\bigl(\ix{U(t_0+\ep,r)Q(r)U(t_0+\ep,r)^\star h}{h}-\ix{U(t_0,r)Q(r)U(t_0,r)^\star h}{h}\bigr)\,dr\nonumber \\
&+\int_{t_0}^{t_0+\ep}\ix{U(t_0+\ep,r)Q(r)U(t_0+\ep,r)^\star h}{h}\,dr\biggr]=:\frac{I_\ep+J_\ep}{\ep}.\nonumber
\end{align}
Since for every $h\in D(A(t)^\star)$ the mapping $$r\in[t_0,t_0+\ep]\longmapsto \ix{U(t_0+\ep,r)Q(r)U(t_0+\ep,r)^\star h}{h}\in\R$$ is continuous, by the Mean Value Theorem for integrals there exists $r_\ep\in(t_0,t_0+\ep)$ such that
\begin{equation}
\frac{J_\ep}{\ep}=\ix{U(t_0+\ep ,r_\ep)Q(r)U(t_0+\ep,r_\ep)^\star h}{h}. \nonumber
\end{equation}
Since for every $h\in D(A(t)^\star)$ and $s\leq t_0$, the mapping $$(\ep, r)\in[s-t_0,+\infty)\times[t_0,t_0+\ep]\longmapsto \ix{U(t_0+\ep,r)Q(r)U(t_0+\ep,r)^\star h}{h}\in\R$$ is continuous, we have
\begin{equation}
\lim_{\ep\rightarrow 0^+}\frac{J_\ep}{\ep}=\lim_{\ep\rightarrow 0^+}\left(\ix{U(t_0+\ep ,r_\ep)Q(r)U(t_0+\ep,r_\ep)^\star h}{h}\right)=\norm{\Qp{t}h}_X^2. \nonumber
\end{equation}
Let us consider $I_\ep$. We have
\begin{align}
\frac{I_\ep}{\ep}&=\frac{1}{\ep}\biggl(\int_s^t\ix{\left(U(t+\ep,r)Q(r)U^\star(t+\ep,r)-U(t+\ep,r)Q(r)U^\star(t,r)\right)h}{h}\,dr\nonumber\\
&+\int_s^t\ix{\left(U(t+\ep,r)Q(r)U^\star(t,r)-U(t,r)Q(r)U^\star(t,r)\right)h}{h}\,dr\biggr)\nonumber 
%&=\frac{1}{\ep}\int_s^t\ix{\left(U(t+\ep,r)Q(r)U^\star(t+\ep,r)-U(t+\ep,r)Q(r)U^\star(t,r)\right)h}{h}\,dr\nonumber\\
%&+\int_s^t\ix{\left(U(t,r)Q(r)U^\star(t+\ep,r)-U(t,r)Q(r)U^\star(t,r)\right)h}{h}\,dr.\nonumber
\end{align}
By the Dominated convergence Theorem  we get
\begin{equation}
\lim_{\ep\rightarrow 0^+}\frac{I_\ep}{\ep}=2\ix{Q(t,s)A(t)^\star h}{h},
\end{equation}
and \eqref{derivataqts} follows. \eqref{derivataqts2}is an immediate consequence of the Fundamental Theorem of  Calculus.
\end{proof}

\begin{thm}
Assume that Hypothesis \ref{2bis} holds true. For every $(s,t)\in\Delta$ and $\ph\in\mathfrak{B}^2_t(X)$ we have
\begin{align}
\label{st-bohr}&\pst\varphi\in\mathfrak{B}^2_s(X),\\
&\frac{\de}{\de s} \pst \ph (x)=-L(s)\pst \ph (x),\quad x\in X,\label{derivata-sx}\\
&\frac{\de}{\de t} \pst \ph (x)=\pst L(t) \ph (x),\quad x\in X.\label{derivata-dx}
\end{align}
\end{thm}
\begin{proof}  
%So \eqref{derivata-sx} and \eqref{derivata-dx} for $\varphi\in\mathcal{E}_t(X)$ follow by applying the chain rule to \eqref{xiAsemigruppo} and remembering \eqref{st}, \eqref{derivataqts}, \eqref{derivataqts2} and \eqref{boooh}.
%Let now prove \eqref{derivata-sx} and \eqref{derivata-dx} for $\ph\in\mathfrak{B}^2_t(X)$ of the form
Let $\ph\in\mathfrak{B}^2_t(X)$. Then
\[
\varphi(x):=\psi(\mathfrak{I}_n^tP_n^tx),\quad x\in X,
\]
where $n\in\Nat$, $\psi\in AP^2(\R^n)$ and the operators $\mathfrak{I}_n^r$ and $P_n^r$ are defined in \eqref{proR} and \eqref{isoR}, respectively. Let $\{\ph_k\}_{k\in\Nat}\subseteq\mathcal{E}_t(X)$ be the approximation built in the proof of Proposition \ref{approssimazionechit}, namely
\[
\ph_k(x):=\psi_{k}(\mathfrak{I}_n^tP_n^tx), \quad x\in X,
\]
where $\psi_{k}\in {\rm Trig}(\R^n)$.
We recall that since $\psi_{k}\in {\rm Trig}(\R^n)$ then $\ph_k\in \mathcal{E}_t(X)$, for every $k\in\Nat$ (see \eqref{prova}). Let $h\in \R^n$ we set $$\Phi^{(h)}(\xi)=e^{\iprod{h}{\xi}_{\R^n}},\ \ \xi\in\R^n.$$   So by \eqref{prova} and \eqref{st} we have
\begin{align*}
\pst\Phi^{(h)}(\mathfrak{I}_n^tP_n^r x)=e^{-\frac{1}{2}\ix{Q(t,s) h}{h}}e^{i\ix{x}{U(t,s)^\star {\mathfrak{I}_n^t}^\star  h}},\qquad x\in X.
\end{align*}
Recalling that ${\mathfrak{I}_n^t}^\star  h\in D(A(t)^\star )$ and that $U(t,s)^\star (D(A(t)^\star ))\subseteq D(A(s)^\star )$, for every $k\in\Nat$ we deduce 
\begin{equation*}
\pst\ph_k\in\mathcal{E}_s(X),
\end{equation*}
so by the first part of this proof for every $k\in\Nat$ we have
\begin{align}
&\frac{\de}{\de s} \pst \ph_k(x)=-L(s)\pst \ph_k (x),\quad x\in X,\label{derivata-sxn}\\
&\frac{\de}{\de t} \pst \ph_k(x)=\pst L(t) \ph_k (x),\quad x\in X.\label{derivata-dxn}
\end{align}
So by \eqref{convergenza-grafico}  letting $k\rightarrow +\infty$ in \eqref{derivata-sxn} and \eqref{derivata-dxn} we obtain \eqref{derivata-sx} and \eqref{derivata-dx}, respectively. 
\end{proof}

\section{Invariant measures}\label{misureinvarianti}

In this section we will investigate existence and uniqueness of an evolution system of measures for $\{\pst\}_{(s,t)\in \overline{\Delta}}$ in $\R$. 

\begin{dfn}
An evolution system of measures for $\{\pst\}_{(s,t)\in \overline{\Delta}}$ in $\R$ is a family of Borel probability measures $\{\nu_r\}_{r\in \R}$ such that 
\begin{equation}\label{invarianzadef}
\int_XP_{s,t}\ph(x)\nu_s(dx)=\int_X\ph(x)\nu_t(dx),\ \ s\leq t,\ \ph\in C_b(X).
\end{equation}
\end{dfn}

First of all we prove the following useful characterization of evolution system of measures for $\{\pst\}_{(s,t)\in \overline{\Delta}}$ in $\R$.
\begin{prop}\label{prop5.2}
Assume that Hypothesis \ref{1} holds true. Then $\{\nu_r\}_{r\in\R}$ is an evolution system of measures for $\{\pst\}_{(s,t)\in \overline{\Delta}}$ in $\R$ if and only if 
\begin{equation}\label{invarianza}
\rwhat{\nu_t}(h)=e^{-\frac{1}{2}\ix{Q(t,s) h}{h}}\rwhat{\nu_s}(\uts^\star h),\ \  s\leq t.
\end{equation}
\end{prop}

\begin{proof}
Let $h\in X$ and let $s\leq t$. If $\{\nu_r\}_{r\in\R}$ is an evolution system of measures for $\{\pst\}_{(s,t)\in \overline{\Delta}}$ in $\R$, then we have
\begin{align*}
\rwhat{\nu}_t(h)&=\int_X e^{i\ix{h}{x}}\,\nu_t(dx)=\int_X \pst (e^{i\ix{h}{\cdot}})(x)\,\nu_s(dx)\\
&=\int_X\int_X e^{i\ix{h}{y}}\,\Ng_{U(t,s)x , Q(t,s)}(dy)\nu_s(dx)=\int_X\rwhat{\Ng}_{U(t,s)x , Q(t,s)}(h)\nu_s(dx)\\
&=\int_X e^{i\ix{h}{U(t,s)x }}e^{-\frac{1}{2}\ix{Q(t,s) h}{h}}\nu_s(dx)=e^{-\frac{1}{2}\ix{Q(t,s) h}{h}}\int_Xe^{i\ix{h}{U(t,s)x}}\,\nu_s(dx)\\
&=e^{-\frac{1}{2}\ix{Q(t,s) h}{h}}\rwhat{\nu}_s(\uts^\star h),
\end{align*}
and \eqref{invarianza} holds.

Conversely, we assume that \eqref{invarianza} holds true. Given $h\in X$, we first show that \eqref{invarianzadef} holds for $\ph(x)=e^{i\ix{h}{x}}$, $x\in X$.  If $s\leq t$, we have by definition
\begin{equation}\label{inv1}
\rwhat{\nu}_t(h)=\int_X e^{i\ix{h}{x}}\,\nu_t(dx),
\end{equation}
and by \eqref{invarianza} we have
\begin{align}\label{inv2}
\rwhat{\nu}_t(h)&= e^{-\frac{1}{2}\ix{Q(t,s) h}{h}}\rwhat{\nu}_s(\uts^\star h)=e^{-\frac{1}{2}\ix{Q(t,s) h}{h}}\int_Xe^{i\ix{h}{U(t,s)x}}\,\nu_s(dx)\nonumber\\
&=\int_X e^{i\ix{h}{U(t,s)x }}e^{-\frac{1}{2}\ix{Q(t,s) h}{h}}\nu_s(dx)\nonumber\\
&=\int_X\int_X e^{i\ix{h}{y}}\,\Ng_{U(t,s)x , Q(t,s)}(dy)\nu_s(dx)=\int_X\rwhat{\Ng}_{U(t,s)x , Q(t,s)}(h)\nu_s(dx)\nonumber\\
&=\int_X \pst (e^{i\ix{h}{\cdot}})(x)\,\nu_s(dx).
\end{align}
By \eqref{inv1} and \eqref{inv2}, \eqref{invarianzadef} holds for $\ph(x)=e^{i\ix{h}{x}}$. 

If $\ph\in C_b(X)$ by Proposition \ref{approssimazione} there exist a 2-sequence $\{\varphi_{n,m}\}_{n,m\in\Nat}\subseteq \mathcal{E}_{r}(X)$, $\{\varphi_{n}\}_{n\in\Nat}\subseteq \mathcal{F}_{r}C_b(X)$ and $\{c_n\}_{n\in\Nat}\subseteq [0,+\infty)$ such that 
\begin{align}
&\norm{\varphi_{n,m}}_{C_b(X)}\leq c_n\norm{\varphi_n}_{C_b(X)},\qquad \  n,m\in\Nat,\label{1d}\\
&\norm{\varphi_{n}}_{C_b(X)}\leq\norm{\varphi}_{C_b(X)},\qquad \quad \quad \quad n\in\Nat,\label{2d}\\
&\lim_{m\rightarrow +\infty}\abs{\varphi_{n,m}(x)-\varphi_n(x)}=0,\qquad \ n\in\Nat,\ x\in X,\label{3d}\\
&\lim_{n\rightarrow +\infty}\abs{\varphi_{n}(x)-\varphi(x)}=0,\qquad \quad\ \ \ x\in X.\label{4d}
\end{align}
Of course $\ph_{n,m}$ satisfies \eqref{invarianzadef} for every $n,m\in\Nat$. To prove that $\ph$ satisfies \eqref{invarianzadef} we apply two times the dominated convergence theorem: the first time as $m\rightarrow+\infty$ thanks to \eqref{1d} and \eqref{3d} and the second as $n\rightarrow+\infty$ thanks to \eqref{2d} and \eqref{4d}.

\end{proof}

\begin{oss}\label{invdet} We consider the problem 
\begin{equation}
\begin{cases}
\dfrac{\partial u}{\partial t}(t,x)=A(t)u(t,\cdot)(x),\ \  s<t\\
u(s,x)=y\in X
\end{cases}
\end{equation}
whose corresponding transition evolution operator  $\{V_{s,t}\}_{(s,t)\in \overline{\Delta}}$ is given by 
\begin{equation}\label{Vst}
V_{s,t}\ph(x)=\ph(U(t,s)x),\ \ \ph\in C_b(X),\ x\in X,\ s\leq t.
\end{equation}
In this case \eqref{invarianzadef} reads as
\begin{equation}
\int_X\ph(U(t,s)x)\, \mu_s(dx)=\int_X\ph(x)\, \mu_t(dx),\ \  s\leq t
\end{equation}
and by Proposition \ref{prop5.2} $\{\mu_r\}_{r\in\R}$ is an evolution system of measures for $\{V_{s,t}\}_{(s,t)\in \overline{\Delta}}$ in $\R$ if and only if 
\begin{equation}\label{invarianzaU}
\rwhat{\mu}_t(h)= \rwhat{\mu}_s(U(t,s)^\star h),\ \ s\leq t,\ h\in X.
\end{equation}
\end{oss}

We recall the infinite dimensional version of the Bochner theorem for the characteristic functions of Borel probability measures in separable Hilbert spaces (see \cite[Thm 2.27 pag. 61]{MR3236753}).

\begin{thm}[Bochner]\label{bochner} Let X be a separable Hilbert space and let $\ph:X\longrightarrow\mathbb{C}$.  $\ph$ is the characteristic function of a probability measure $\mu$ on $(X,\mathcal{B}(X))$ if and only if
\begin{enumerate}
\item[(1)] $\ph$ is continuous and $\ph(0)=1$,
\item[(2)] $\ph$ is a positive definite function, namely for every $k\in\Nat$ and every choice of $x_1,...,x_k\in X$ and $c_1,..,c_k\in\mathbb{C}$ we have
\begin{equation}
\sum_{i,j=1}^k c_i\overline{c_j}\ph(x_i-x_j)\geq 0,
\end{equation}
\item[(3)]for every $\ep>0$ there exists a non-negative nuclear operator $S_\ep$ such that 
\begin{equation}
1-\mathfrak{Re}\ph(x)\leq \ep,
\end{equation}
for all $x\in X$ satisfying $\ix{S_\ep x}{x}\leq 1$.
\end{enumerate} 
\end{thm}
\begin{oss}
Condition (3) is related to the continuity with respect to the Sazonov topology that is relevant in our setting (see \cite[Thm\,7.13.7]{MR2267655}). If $X=\R^n$ then condition (3) is verified and Theorem \ref{bochner} is the Bochner Theorem (see \cite[Thm.\,7.13.1]{MR2267655}).
\end{oss}

To prove Theorem \ref{esistenza} we need the following proposition that should be known but we were not able to find any reference. For this reason we include a proof of this result.

 We say that a sequence $\{L_k\}_{k\in\Nat}\subseteq {\mathcal L}(X)$ is increasing if $L_{k+1}-L_k$ is nonnegative for any $k\in\Nat$, i.e., $\langle (L_{k+1}-L_k)x,x\rangle \geq 0$ for every $x\in X$ and for every $k\in\Nat$.  

\begin{prop}\label{increasingConvergence}
Let $\{T_k\}_{k\in\Nat}\subseteq {\mathcal L}_1(X)$ be an increasing sequence of self-adjoint non-negative operators having finite trace such that $\displaystyle{\sup_{k\in\Nat}\|T_k\|_{{\mathcal L}_1(X)}<\infty}$. Then, there is $T\in{\mathcal L}_1(X)$ such that $T_k\xrightarrow{k\rightarrow+\infty} T$ in ${\mathcal L}_1(X)$.
\end{prop}
\begin{proof}
We set $$\al:=\sup_{k\in\Nat}\|T_k\|_{\mathcal {L}_1(X)}$$ and we first show that there exists $T\in{\mathcal L}(X)$ such that $T_k\xrightarrow{k\rightarrow+\infty} T$ in ${\mathcal L}(X)$. 
By hypothesis, we have 
\[
\eta(x):=\lim_{k\to\infty}\ix{x}{T_kx}\leq \al
\]
whenever $x\in X$ and $\norm{x}_X\leq 1$. Set now $T_{nm}=T_n-T_m$ for $m<n$, so that the $L_{nm}$ are non-negative. For every $x,y\in X$ we have 
\begin{align*}
|\langle y,T_{nm}x\rangle_X| &= 
|\langle T_{nm}^{\frac{1}{2}}x,T_{nm}^{\frac{1}{2}}y\rangle_X| \leq 
\|T_{nm}^{\frac{1}{2}}x\|_X\|T_{nm}^{\frac{1}{2}}y\|_X
\\
&= \langle x,T_{nm}x\rangle_X
\langle y,T_{nm}y\rangle_X\leq 
\eta(y)\langle x,T_{nm}x\rangle_X.
\end{align*}
Therefore,
\begin{equation}\label{main}
\|T_{nm}x\| _X= \sup_{\|y\|_X\leq 1} |\langle y,T_{nm}x\rangle_X|
\leq \al \langle x,T_{nm}x\rangle_X\quad \text{for every \ }n,m\in\Nat.
\end{equation}
We claim that the right hand side of \eqref{main} vanishes as $n,m\longrightarrow+\infty$.
We argue by contradiction and we assume that there exist $x\in X$ and $\be>0$ such that for every $m\in\Nat$ there exists $p(m)\in\Nat$ such that $\langle(T_{m+p(m)}-T_m)x,x\rangle_X\geq \be>0$. Then, defining $m_1=1,p_1=p(m_1)$ and by iteration $m_{j+1}=m_j+p_j$ we would have 
\[
\langle (T_{m_N+p_N}-T_1)x,x\rangle_X = 
\sum_{j=1}^N\langle(T_{m_j+p_j}-T_{m_j})x,x\rangle _X\geq N\be\longrightarrow +\infty, \quad \text{as}\ N\longrightarrow+\infty
.\]
which contradicts $\eta(x)\leq \al$ for every $x\in X$. Summarizing, $\{T_kx\}_{k\in\Nat}$ is a Cauchy sequence for every $x\in X$ and defines the self-adjoint nonnegative limit operator $\displaystyle{Tx=\lim_{k\rightarrow+\infty}T_kx}$. 

Let us show that $T\in{\mathcal L}_1(X)$ and that $T_k\xrightarrow{k\rightarrow+\infty}  T$ in ${\mathcal L}_1(X)$. For an arbitrary Hilbert basis $\{e_k\}_{k\in\Nat}$ by the Fatou lemma applied to the series we have  
\[
{\rm Tr}(T) = 
\sum_{k=1}^\infty \langle\lim_{n\to\infty}T_ne_k,e_k\rangle_X
\leq \liminf_{n\to\infty}\sum_{k=1}^\infty \langle T_ne_k,e_k\rangle_X=\liminf_{n\to\infty}{\rm Tr}(T_n)
\]
and $
T\in{\mathcal L}_1(X)$. To prove the convergence, notice that by monotonicity $\langle T_ke_n,e_n\rangle_X\leq \langle Te_n,e_n\rangle_X$ for every $k,n\in\Nat$, whence the thesis follows by the Dominated Convergence Theorem applied to series
\[
{\rm Tr}\, (T-T_k) = \sum_{n=1}^\infty \langle (T-T_k)e_n,e_n\rangle \longrightarrow 0, \quad \text{as }k\longrightarrow+\infty
.\]  
\end{proof}

\begin{thm} \label{esistenza}
Assume that Hypothesis \ref{1} holds true. If for every $t\in\R$
\begin{equation}\label{tracebound}
\sup_{s<t}\left[\Tr{Q(t,s)}\right]<+\infty
\end{equation}
then the operator $$Q(t,-\infty):=\int_{-\infty}^tU(t,r) Q(r)U(t,r)^\star dr$$ is well defined and it has finite trace for every $t\in\R$ . The family of measures $\{\, \g_t\}_{t\in\R}$ given by
\begin{equation}\label{gammat}
\g_t:=\Ng_{0,Q(t,-\infty)},\ \ t\in\R,
\end{equation}
is an evolution system of measures for $\{\pst\}_{(s,t)\in \overline{\Delta}}$ in $\R$. 

Moreover $\{\nu_t\}_{t\in\R}$ is an evolution system of measures for $\{\pst\}_{(s,t)\in \overline{\Delta}}$ in $\R$ if and only if 
\[
\nu_t=\, \g_t\star\mu_t,\ \ t\in \R,
\]
where $\{\mu_t\}_{t\in\R}$ is an evolution system of measures  for $\{V_{s,t}\}_{(s,t)\in \overline{\Delta}}$ in $\R$ and  $\{V_{s,t}\}_{(s,t)\in \overline{\Delta}}$ is given by \eqref{Vst}

\end{thm}

\begin{proof}
 By Proposition \ref{increasingConvergence} $Q(t,-\infty)$ is well defined and it belongs to $\op^+_1(X)$. 
Now we show that the measures $\g_t$ satisfy \eqref{invarianza}. Let $s\leq t$ and $h\in X$, then
\begin{align}
&U(t,s)Q(s,-\infty)U(t,s)^\star h=\int_{-\infty}^s U(t,s)U(s,r)Q(r)U^\star(s,r)U(t,s)^\star h\,dr\nonumber\\
&=\int_{-\infty}^s U(t,r)Q(r)U(t,r)^\star h\,dr=Q(t,-\infty)h-Q(t,s)h.\nonumber
\end{align}
Moreover
\begin{align}
\rwhat{\, \g}_t(h)&=e^{-\frac{1}{2}\ix{Q(t,-\infty)h}{h}}\nonumber \\
\rwhat{\g}_t(U(t,s)^\star h)&=e^{-\frac{1}{2}\ix{Q(s,-\infty)U(t,s)^\star h}{U(t,s)^\star h}}.\nonumber
\end{align}
Therefore, we have
\begin{align}\label{qts}
\ix{Q(s,-\infty)U(t,s)^\star h}{U(t,s)^\star h}&=\ix{U(t,s)Q(s,-\infty)U(t,s)^\star h}{h}\nonumber \\
&=\ix{Q(t,-\infty)h-Q(t,s)h}{h}.
\end{align}
Hence, we get
\begin{equation}
\rwhat{\g}_s(U(t,s)^\star h)=e^{i\ix{g(t,-\infty)-g(t,s)}{h}}e^{-\frac{1}{2}\ix{Q(t,-\infty)h}{h}}e^{\frac{1}{2}\ix{Q(t,s)h}{h}}
\end{equation}
and
\begin{equation}
\rwhat{\, \g}_t(h)=e^{-\frac{1}{2}\ix{Q(t,s) h}{h}}\rwhat{\g_s}(U(t,s)^\star h).
\end{equation}
Now we prove the last statement. Let $\{\mu_t\}_{t\in\R}$ be an evolution system of measures in $\R$ for $V_{s,t}$, we set $\nu_t=\g_t\star\mu_t$ and we show that $\{\nu_t\}_{t\in\R}$ is an evolution system of measures in $\R$ for $\pst$, namely we show that for all $\nu_t$ satisfies \eqref{invarianza} for all $t\in\R$. 
Let $t\in\R$ and $h\in X$, then for every $s\leq t$ by \eqref{qts} we have
\begin{align}
&\rwhat{\nu}_t(h)=\rwhat{\g}_t(h)\rwhat{\mu}_t(h)=e^{-\frac{1}{2}\ix{Q(t,-\infty)h}{h}}\rwhat{\mu}_s(U(t,s)^\star h)\nonumber \\
&=e^{-\frac{1}{2}\ix{Q(t,s)h}{h}}e^{-\frac{1}{2}\ix{Q(s,-\infty)U(t,s)^\star h}{U(t,s)^\star h}}\rwhat{\mu}_s(U(t,s)^\star h)\nonumber \\
%&=e^{-\frac{1}{2}\ix{Q(t,s)h}{h}}e^{i\ix{U(t,s)g(s,-\infty)}{h}}e^{-\frac{1}{2}\ix{U(t,s)Q(s,-\infty)U(t,s)^\star h}{h}}\rwhat{\mu}_s(U(t,s)^\star h)\nonumber\\
%&=e^{-\frac{1}{2}\ix{Q(t,s)h}{h}}e^{i\ix{g(s,-\infty)}{U(t,s)^\star h}}e^{-\frac{1}{2}\ix{Q(s,-\infty)U(t,s)^\star h}{U(t,s)^\star h}}\rwhat{\mu}_s(U(t,s)^\star h)\nonumber\\
&=e^{-\frac{1}{2}\ix{Q(t,s)h}{h}}\rwhat{\nu}_s(U(t,s)^\star h).\nonumber
\end{align}
Conversely, if $\{\nu_t\}_{t\in\R}$ is an evolution system of measures in $\R$  for $\pst$,
then 
\begin{equation}
\rwhat{\nu}_t(h)=e^{-\frac{1}{2}\ix{Q(t,s) h}{h}}\rwhat{\nu}_s(U(t,s)^\star h),\ \  h\in X,\ s\leq t.
\end{equation}
We set 
\begin{equation}\label{psit}
\psi_t(h)=\lim_{s\rightarrow-\infty}\rwhat{\nu}_s(U(t,s)^\star h)=e^{\frac{1}{2}\ix{Q(t,-\infty) h}{h}}\rwhat{\nu}_t(h),
\end{equation}
and so
\begin{equation}\label{convoluzionemunu}
\rwhat{\nu}_t(h)=e^{-\frac{1}{2}\ix{Q(t,-\infty) h}{h}}\lim_{s\rightarrow-\infty}\rwhat{\nu}_s(U(t,s)^\star h)=\rwhat{\g}_t(h)\psi_t(h).
\end{equation}
To conclude the proof is sufficient to show that
\begin{enumerate}
\item[(a)] for all $t\in\R$, $\psi_t$ is the characteristic function of a Borel probability measure $\mu_t$;
\item[(b)] the family $\{\mu_t\}_{t\in\R}$ is an evolution system of measure for $V_{s,t}$.
\end{enumerate}
To prove (a), we note first that $e^{-i\ix{g(t,-\infty)}{h}}\rwhat{\nu}_t(h)$ is the characteristic function of the measure $\zeta_t:=\delta_{g(t,-\infty)}\star \nu_t$ . Then, we apply Theorem \ref{bochner} to $\zeta_t$ and we obtain that 
\begin{enumerate}
\item $\rwhat{\zeta}_t(0)=1$,
\item $\rwhat{\zeta}_t$ is positive definite,
\item for every $\ep>0$ there exists a non-negative nuclear operator $S_\ep$ such that 
\begin{equation}%\label{epboch}
1-\mathfrak{Re}\,\rwhat{\zeta}_t(h)\leq \ep,
\end{equation}
for all $h\in X$ satisfying $\ix{S_\ep h}{h}\leq 1$.
\end{enumerate}
Since $$\psi_t(h)=\frac{\rwhat{\nu}_t(h)}{\rwhat{\, \g}_t(h)}=e^{\frac{1}{2}\ix{Q(t,-\infty) h}{h}}\rwhat{\zeta}_t(h),\quad h\in X,$$ then $\psi_t(0)=1$, $\psi_t$ is positive definite and
\begin{equation}
1-\mathfrak{Re}\,\psi_t(h)\leq 1-\mathfrak{Re}\,\rwhat{\zeta}_t(h)\leq\ep,
\end{equation}
for all $h\in X$ satisfying $\ix{S_\ep h}{h}\leq 1$. Hence, by theorem \ref{bochner} for all $t\in\R$ there exists a Borel probability measure $\mu_t$ such that $\rwhat{\mu}_t\equiv \psi_t$ and for all $h\in X$ we get
\begin{align}
\rwhat{\mu}_s(\uts^\star h)&= \psi_s(\uts^\star h)=\lim_{\si\rightarrow-\infty}\rwhat{\nu}_r(U(s,\si)^\star\uts^\star h)\nonumber \\
&=\lim_{\si\rightarrow-\infty}\rwhat{\nu}_r(U(t,\si)^\star h)=\rwhat{\mu}_t(h). \nonumber
\end{align}
Hence $\{\mu_t\}_{t\in\R}$ is an evolution system of measure for $V_{s,t}$ and the statement follows.
\end{proof}

\begin{thm}\label{unicità}
Assume that Hypothesis \ref{1} holds true. If for every $t\in\R$
\begin{equation*}
\sup_{s<t}\left[\Tr{Q(t,s)}\right]<+\infty
\end{equation*}
and for every $t\in\R$ and $x\in X$
\begin{equation}\label{ergoU}
\lim_{s\rightarrow-\infty}U(t,s)x=0,
\end{equation}
then for every $f\in C_b(X)$ we have
%\begin{equation}\label{comp1}
%\lim_{t\rightarrow+\infty}\biggl(\pst f(x)-\int_Xf(y)\, \, \g_t(dy)\biggr)=0,\quad s\in\R,\; x\in X
%\end{equation}
%\textcolor{red}{questa non l'ho dimostrata e ho il serio dubbio che sia falsa. Ai nostri scopi non serve}
%and
\begin{equation}\label{comp2}
\lim_{s\rightarrow-\infty}\pst f(x)=m_t(f):=\int_Xf(y)\, \, \g_t(dy),\quad t\in\R,\; x\in X.
\end{equation}
where $\{\g_t\}_{t\in\R}$ is the evolution system of measures for $\{\pst\}_{(s,t)\in\overline{\Delta}}$ in $\R$ given by \eqref{gammat}.
\end{thm}
\begin{proof}
%Let $\{\nu_t\}_{t\in\R}$ be an evolution system of measures in $\R$ invariant for $\pst$. Since  $\displaystyle{\lim_{s\rightarrow-\infty}U(t,s)x=0}$, we apply the Dominated Convergence Theorem to \eqref{invarianza} and we obtain
%\begin{equation}
%\rwhat{\nu_t}(h)=e^{-\frac{1}{2}\ix{Q(t,-\infty) h}{h}},\ \  s\leq t.
%\end{equation}
%Hence $\nu_t=\g_t$ for all $t\in\R$.
The statement follows by \cite[Example 3.8.15 page 135]{MR1642391} since $U(t,s)x \xrightarrow{s\rightarrow-\infty}0$ in $X$ and $Q(t,s)\xrightarrow{s\rightarrow-\infty} Q(t,-\infty)$ in $\op_1(X)$.
\end{proof}
\begin{oss}\label{omeganegativo}
We note that if $\{U(t,s)\}_{(s,t)\in\overline{\Delta}}$ verifies \eqref{omega} with $\zeta>0$ then it verifies \eqref{ergoU} and by \cite[Cor. 4.12]{ouy-roc2016} $\{\g_t\}_{t\in\R}$ is the unique evolution system of measures uniformly tight for $\{\pst\}_{(s,t)\in\overline{\Delta}}$.  Moreover, if \eqref{ergoU} does not hold then there may exist many evolution systems of measures for $\pst$ in $\R$, see Remark \ref{nonunico}.
\end{oss}

\begin{cor}\label{prop63}
Assume that Hypothesis \ref{1} holds true. If for every $t\in\R$
\begin{equation*}
\sup_{s<t}\left[\Tr{Q(t,s)}\right]<+\infty
\end{equation*}
and for every $t\in\R$ and $x\in X$
\[
\lim_{s\rightarrow-\infty}U(t,s)x=0,
\]  
then for every $\ph\in C_b(X)$ with strictly positive infimum and $t\in\R$ we have
\begin{equation}
\lim_{s\rightarrow-\infty}\int_X\pst \ph(x)\log\pst \ph(x) \g_t(dy)=m_t(\ph)\log m_t(\ph).
\end{equation}
\end{cor}
\begin{proof}
Let $\ph\in C_b(X)$ have strictly positive infimum. Then $\pst \ph$ belongs to $C_b(X)$ and has positive infimum. Since the mapping $y\longmapsto y\log y$ is $\frac{1}{2}$- H\"older continuous on bounded sets of $(0,+\infty)$, we get
\begin{align}
&\abs{\int_X\pst \ph(x)\log\pst \ph(x)\, \g_t(dy)-m_t(\ph)\log m_t(\ph)}\nonumber\\
&=\abs{\int_X\left(\pst \ph(x)\log\pst \ph(x)-m_t(\ph)\log m_t(\ph) \right)\,\g_t(dy)}\nonumber\\
&\leq C\int_X\abs{\pst \ph(x)-m_t(\ph)}^{\frac{1}{2}}\,\g_t(dy),\nonumber
\end{align}
for some constant $C>0$. By the Dominated Convergence Theorem and \eqref{comp2} the statement follows. 
\end{proof}

\section{Logarithmic Sobolev inequalities}\label{log-sobolev}

In this Section we need all the results proved in the previous sections, so we assume the following hypothesis.
\begin{ip}\label{4} 
Assume that Hypothesis \ref{2bis} holds true and that \eqref{tracebound} and \eqref{ergoU} are satisfied.
\end{ip}
Let $\{\g_t\}_{t\in{\R}}$ be the evolution system of measures for $\{\pst\}_{(s,t)\in\overline{\Delta}}$ given by \eqref{gammat}.
\begin{lem}\label{lemma31} Assume that Hypothesis \ref{4} holds true. Let $(s,t)\in\Delta$ and let $\ph:[s,t]\times X\rightarrow \R$ be such that 
\begin{enumerate}
\item for every $x\in X$ the function $(s,t)\ni r\rightarrow \ph(r,x)$ is differentiable;
\item for every $r\in[s,t]$ $\ph(r,\cdot)\in \mathfrak{B}^2_r(X)$ and there exist $C>0$ and $m\in\Nat$ such that 
$$\dfrac{\partial\ph}{\partial r}(r,x)\leq C(1+\norm{x}_X^m),\ \ r\in(s,t),\ x\in X.$$
\end{enumerate}
Then 
\begin{align} \label{der31}
\frac{d}{dr}\int_X \ph(r,x)\,\g_r(dx)=\int_X\left(L(r)\ph(r,x)+\frac{\de}{\de r}\ph(r,x)\right)\,\g_r(dx),\quad r\in (s,t).
\end{align}
\end{lem}
\begin{proof}
\eqref{der31} is an immediate consequence of \eqref{derivata-dx}, indeed by the Dominated Convergence Theorem and the Fernique Theorem we get
\begin{align}
\frac{d}{dr}\int_X \ph(r,x)\,\g_r(dx)&=\int_X\frac{d}{dr} P_{s,r}\ph(r,x)\,\g_s(dx)\nonumber \\
&=\int_X \left(P_{s,r}L(r)\ph(r,x)+\left(P_{s,r}\frac{\de}{\de r}\ph(r,\cdot)\right)(x)\right)\,\g_s(dx)\nonumber\\
&=\int_X \left(L(r)\ph(r,x)+\frac{\de}{\de r}\ph(r,x)\right)\,\g_r(dx).\nonumber
\end{align}\end{proof}

Now we can prove one of the main result of this paper.

\begin{thm}\label{LOG}
Assume that Hypothesis \ref{4} holds true. Moreover we assume that $U(t,s)\ac_s\subseteq\ac_t$ for every $(s,t)\in\Delta$ and that there exist $C,\eta>0$ and $\displaystyle{\al\in \left[0,\frac{1}{2}\right)}$ such that
\begin{equation}\label{cond-logsob2}
\norm{U(t,s)_{|_{\ac_s}}}_{\op(\ac_s;\ac_t)}\leq C\, \frac{e^{-\eta(t-s)}}{(t-s)^\al},\qquad\forall\; (s,t)\in\Delta.
\end{equation}
Then, for every $\ph\in C^1_b(X)$, $t\in\R$ and $p\in (1,+\infty)$ we have
\begin{equation}\label{logsobolev}
\int_X|\ph|^p\log\left(|\ph|^p\right)\, d\g_t- m_t(|\ph|^p)\log\left(m_t\left(|\ph|^p\right)\right)\leq\kappa p^2\int_X |\ph|^{p-2}\norm{\Qp{t}\nabla \ph}^2\, \mathbbm{1}_{\{\ph\neq 0\}}d\g_t.
\end{equation}
where 
\begin{equation}\label{logsobcost}
\kappa=C(2\eta)^{2\al-1}\Gamma(1-2\al),
\end{equation}
and $\Gamma$ is the Euler Gamma function.
\end{thm}
\begin{proof}
Let $p\in (1,+\infty)$ and $r\in\R$. We prove first that \eqref{logsobolev} holds for $\ph\in \mathfrak{B}^2_t(X)$ such that $\displaystyle{\inf_{x\in X}\ph(x)>\ep}$ for some $\ep>0$. 
Let $(s,t)\in\Delta$ and $s\leq r\leq t$;  we define a mapping $\psi:[s,t]\times X\rightarrow\R$ by
\begin{align}\label{mappag}
\psi(r,x):=\left(P_{r,t}\ph^p\right)(x).
\end{align}
By \eqref{st-bohr}, $\psi(r,\cdot)\in\mathfrak{B}^2_r(X)$ for every $r\in [s,t]$  and $$\inf_{r\in [s,t]}\inf_{x\in X}\psi(r,x)>\ep^p$$ and thanks to \eqref{derivata-sx} we get
\begin{align}\label{gdr}
\dfrac{\partial \psi}{\partial r}(r,x)=-L(r)\psi(r,\cdot)(x),\quad r\in [s,t],\ x\in X.
\end{align}
Now we consider the function $G:[s,t]\rightarrow \R$ defined by
\[
G(r):=\int_X \psi(r,x)\log(\psi(r,x))\, \g_r(dx).
\]
Since $\psi(r,\cdot)\in \mathfrak{B}^2_r(X)$ has positive infimum independent of $r\in [s,t]$, since  $\log(\cdot):[\ep,+\infty)\rightarrow \R$ has continuous and bounded derivatives of every order, then $\log(\psi(r,\cdot))\in \mathfrak{B}^2_r(X)$ for every $r\in [s,t]$. So by Lemma \ref{lemma31} we obtain
\begin{align}
G'(r)&=\int_X L(r)\left(\psi(r,\cdot)\log(\psi(r,\cdot))\right)(x)\, \g_r(dx)-\int_X L(r)\psi(r,\cdot)(x)\log(\psi(r,x))\, \g_r(dx)\notag\\
&-\int_X L(r)\psi(r,\cdot)(x)\, \g_r(dx)\label{P1}.
\end{align}
By \eqref{ftgeneratore}, for every $r\in [s,t]$ and $F,\Psi,\Phi\in \mathcal{F}_r C^{2}_b(X)$ with $F$ having positive infimum, by \eqref{ftgeneratore} we have
\begin{align}\label{LeibnizL}
L(r)(\Phi\Psi)&=\Phi L(r)\Psi+\Psi L(r)\Phi+\ix{\Qp{r}\nabla\Phi}{\Qp{r}\nabla\Psi}, \\
\label{logaritmoL}
L(r)(\log F)&=\frac{1}{F}L(r)F-\frac{1}{2}\dfrac{\norm{\Qp{r}\nabla F}}{F^2}.
\end{align}
Using \eqref{LeibnizL} in \eqref{P1} we obtain
\begin{align*}
G'(r)&=-\int_X L(r)\psi(r,\cdot)(x)\log(\psi(r,x))\, \g_r(dx)-\int_X L(r)\psi(r,\cdot)(x)\, \g_r(dx)\notag\\
&+\int_X \log \psi(r,x) L(r)\psi(r,\cdot)(x)\, \g_r(dx)+\int_X \psi(r,x) L(r)\log(\psi(r,\cdot))(x)\, \g_r(dx)\notag\\
&+\int_X \ix{\Qp{r}\nabla \psi(r,x)}{\Qp{r}\nabla \log \psi(r,x)}\, \g_r(dx),
\end{align*}
hence by \eqref{logaritmoL} we have
\begin{align}\label{Gprimo}
G'(r)&=-\frac{1}{2}\int_X\frac{\norm{\Qp{r}\nabla \psi(r,x)}^2}{\psi^2(r,x)}\, \g_r(dx)+\int_X \dfrac{\norm{\Qp{r}\nabla \psi(r,x)}^2}{\psi(r,x)}\, \g_r(dx).
\end{align}
Since the first summand in the right hand side of \eqref{Gprimo} is negative, we get
\begin{align*}
G'(r)&\leq \bigint_X \dfrac{\norm{\Qp{r}\nabla \psi(r,x)}^2}{\psi(r,x)}\, \g_r(dx)=\int_X \dfrac{\norm{\Qp{r}\nabla (P_{r,t}\ph^p)(x)}^2}{(P_{r,t}\ph^p)(x)}\, \g_r(dx).
\end{align*}
Applying \eqref{stimachemiserve}, we obtain
\begin{align}\label{P2}
G'(r)&\leq C\dfrac{e^{-2\eta (t-r)}}{(t-r)^{2\alpha}}\int_X \dfrac{\left( P_{r,t}\norm{\Qp{t}\nabla \ph^p}\right)^2(x)}{(P_{r,t}\ph^p)(x)}\, \g_r(dx)
\end{align}
and by the H\"older inequality we get
\begin{equation}\label{Holder}
P_{r,t}\left(\norm{\Qp{t}\nabla \ph^p}\right)\leq \left(P_{r,t}\left(\dfrac{\norm{\Qp{t}\nabla \ph^p}^2}{\ph^p}\right)\right)^{1/2}\left(P_{r,t}\left(\ph^p\right)\right)^{1/2}.
\end{equation}
Applying \eqref{Holder} to \eqref{P2}, by \eqref{invarianzadef} we get
\begin{align*}
G'(r)&\leq C\dfrac{e^{-2\eta (t-r)}}{(t-r)^{2\alpha}}\int_X P_{r,t}\dfrac{\norm{\Qp{t}\nabla \ph^p}^2}{\ph^p}(x)\, \g_r(dx)\\
&=C\dfrac{e^{-2\eta (t-r)}}{(t-r)^{2\alpha}}\int_X \dfrac{\norm{\Qp{t}\nabla \ph^p(x)}^2}{\ph^p(x)}\, \g_t(dx)\\
&=p^2C\dfrac{e^{-2\eta (t-r)}}{(t-r)^{2\alpha}}\int_X \ph(x)^{p-2}\norm{\Qp{t}\nabla \ph(x)}^2\, \g_t(dx).
\end{align*}
Integrating with respect to $r$ over $[s,t]$ we obtain
{\small \begin{align*}
\int_X \ph^p\log \ph^p\, d\g_t-\int_X \pst \ph^p\log \pst \ph^p\, d\g_s\leq p^2\left(C\int^t_{s}\dfrac{e^{-2\eta (t-r)}}{(t-r)^{2\alpha}}dr\right)\int_X \ph^{p-2}\norm{\Qp{t}\nabla \ph}^2\, d\g_t,
\end{align*}}
Letting $s\rightarrow -\infty$ and using Proposition \ref{prop63} we conclude
\begin{align}\label{P3}
\int_X \ph^p\log \ph^p\, \g_t(dx)\leq m_t(\ph^p)m_t(\log \ph^p)+p^2\kappa\int_X \ph^{p-2}\norm{\Qp{t}\nabla \ph}^2\, \g_t(dx),
\end{align}
where
\begin{equation}
\kappa:=C\int^{+\infty}_0\dfrac{e^{-2\eta r}}{r^{2\alpha}}dr=C(2\eta)^{2\al-1}\Gamma(1-2\al).
\end{equation}

We obtain \eqref{logsobolev} for every $\ph\in \mathfrak{B}^2_t(X)$ applying \eqref{P3} to the standard approximation $\displaystyle{\ph_n=\sqrt{\ph^2+\frac{1}{n^2}}}$ and letting $n\rightarrow +\infty$ in  \eqref{P3}. We stress that, fixed $n\in\Nat$, the function $\varphi_n$ belongs to $\mathfrak{B}^2_t(X)$ since the function $h(x)=\sqrt{x^2+\frac{1}{n}}$ has continuous and bounded derivatives of every order. 
Finally, if $\ph\in C^1_b(X)$ by Proposition \ref{approssimazione} there exist a 2-sequence $\{\varphi_{n,m}\}_{n,m\in\Nat}\subseteq \mathcal{E}_{r}(X)$, $\{\varphi_{n}\}_{n\in\Nat}\subseteq \mathcal{F}_{r}C^1_b(X)$ and $\{c_n\}_{n\in\Nat}\subseteq [0,+\infty)$ such that 
\begin{align}
&\norm{\varphi_{n,m}}_{C^1_b(X)}\leq c_n\norm{\varphi_n}_{C^1_b(X)},\ \  \quad\qquad\qquad\qquad\qquad\qquad\qquad n,m\in\Nat,\label{1e}\\
&\norm{\varphi_{n}}_{C^1_b(X)}\leq\norm{\varphi}_{C^1_b(X)},\ \quad\qquad\qquad\qquad\qquad\qquad\qquad \quad \quad \quad n\in\Nat,\label{2e}\\
&\lim_{m\rightarrow +\infty}\left(\abs{\varphi_{n,m}(x)-\varphi_n(x)}+\vert \nabla\varphi_{n,m}(x)-\nabla\varphi_n(x)\vert\right)=0,\qquad \ n\in\Nat,\ x\in X,\label{3e}\\
&\lim_{n\rightarrow +\infty}\left(\abs{\varphi_{n}(x)-\varphi(x)}+\vert \nabla\varphi_{n}(x)-\nabla\varphi(x)\vert\right)=0,\qquad \quad\quad\ \ \ \ \ x\in X.\label{4e}
\end{align}
Noting that $\mathcal{E}_t(X)\subseteq \mathfrak{B}^2_t(X)$, $\ph_{n,m}$ satisfy \eqref{logsobolev} for every $n,m\in\Nat$. We obtain $\ph$ satisfies \eqref{logsobolev} applying two times the Dominated Convergence Theorem: the first time as $m\rightarrow+\infty$ thanks to \eqref{1e} and \eqref{3e} and the second time as $n\rightarrow+\infty$ thanks to \eqref{2e} and \eqref{4e}.
\end{proof}

\begin{oss}
Let us compare the Logarithmic Sobolev Inequality provided by L. Gross in \cite{gro75, gro93, GRO1} with the one in \eqref{logsobolev}. We fix $r\in\R$ and we consider the operator $L(r)$ defined by \eqref{OU}. We assume that $A(r)$ is the infinitesimal generator of a strongly continuous semigroup $\{T^{(r)}(t)\}_{t\geq 0}$ such that
\begin{equation}\label{stima-semi}
\norm{T^{(r)}(t)}_{\mathcal{L}(X)}\leq e^{-c^{(r)}t},\quad t>0,
\end{equation}
for some positive constant $c^{(r)}$ and 
\[
\int^{+\infty}_{0}\Tr{T^{(r)}(t)B(r)B(r)^\star \left(T^{(r)}(t)\right)^\star }dt<+\infty.
\]
Under these assumptions we consider the Ornstein-Uhlenbeck semigroups $\{R^{(r)}(t)\}_{t\geq 0}$, given by
\[
R^{(r)}(t)\varphi(x)=\int_X\varphi(y)\mathcal{N}_{T^{(r)}(t)x,Q^{(r)}_t}(dy),\qquad t>0,\, x\in X,\, \varphi\in C_b(X),
\] 
where 
\[
Q^{(r)}_t:=\int_0^tT^{(r)}(s)B(r)B(r)^\star T^{(r)}(s)^\star ds.
\]
Setting $\displaystyle{Q^{(r)}_\infty:=\int_0^{+\infty}T^{(r)}(s)B(r)B(r)^\star \left(T^{(r)}(s)\right)^\star ds}$, the Gaussian measure $\mu^{(r)}=\mathcal{N}_{0,Q^{(r)}_\infty}$ is the unique invariant measure of $\{R^{(r)}(t)\}_{t\geq 0}$. Moreover it is well known that $\{R^{(r)}(t)\}_{t\geq 0}$ is uniquely extendable to a strongly continuous semigroup in $L^2(X,\mu^{(r)})$, still denoted by $\{R^{(r)}(t)\}_{t\geq 0}$. Its infinitesimal generator is the closure in $L^2(X,\mu^{(r)})$ of the second order Kolmogorov operator $L(r)$ given by \eqref{OU} defined on $\mathcal{E}_r(X)$ (see \cite{Bignamini2021bis}). We still denote it by $L(r)$.

 We note that $\mu^{(r)}$ is not necessarily the Gaussian measure $\gamma_r$ of our evolution system of measures $\{\gamma_t\}_{t\in\R}$ of $\{\pst\}_{(s,t)\in\overline{\Delta}}$. Moreover by \cite{gro75, gro93, GRO1} the measure $\mu^{(r)}$ verifies the following logarithmic Sobolev inequality
\begin{equation}\label{logsobolevA}
\int_X|\ph|^p\log\left(|\ph|^2\right)\, d\mu^{(r)}- m^{(r)}(|\ph|^2)\log\left(m^{(r)}\left(|\ph|^2\right)\right)\leq -\frac{4}{c^{(r)}}\int_X \ph\, L(r)\ph\,d\mu^{(r)}, 
\end{equation}
where $\varphi\in D(L(r))$, $c^{(r)}$ is the constant in \eqref{stima-semi},
\[
m^{(r)}(\psi):=\int_X\psi(x)\mu^{(r)}(dx),\quad \psi\in L^1(X,\mu^{(r)}).
\]

 We recall that under suitable assumptions on $A(r)$ and $B(r)$, we have
\[
\int_X \ph\, L(r)\ph\,d\mu^{(r)}= -\frac{1}{2}\int_X \norm{\Qp{r}\nabla \ph}^2d\mu.
\]
So \eqref{logsobolev} does not coincide in general with \eqref{logsobolevA}.  In the next theorem we will see that  \eqref{logsobolev} implies a hypercontractivity result for $\{\pst\}_{(s,t)\in\overline{\Delta}}$.
\end{oss}

\begin{lem} Assume that there exists a unique evolution system of measures $\{\g_t\}_{t\in\R}$ for $\pst$ in $\R$. Then for any $p\geq 1$ and $(s,t)\in\overline{\Delta}$ the operator $\pst$ is extendable to a linear bounded operator from $L^p(X,\g_t)$ to $L^p(X,\g_s)$. We still denote it by $\pst$. 
\end{lem}
\begin{proof}
For every $f\in C_b(X)$, $(s,t)\in\overline{\Delta}$ and $x\in X$ we have  
\[
|\pst f(x)|^p \leq \int_X |f( y+U(t,s)x  )|^p d\gamma_t = \pst(|f|^p)(x). 
\]
Integrating over $X$ and recalling \eqref{invarianzadef} we obtain 
$$
\int_X |\pst f |^p d\gamma_t \leq  \int_X \pst(|f|^p)\,d\gamma_t = \int_X |f|^p d\gamma_s. $$
Since $C_b(X)$ is dense in $L^p(X, \gamma )$,  $\pst$ has a unique bounded extension still denoted by $\pst$ from the whole $L^p(X, \gamma_t )$ into $L^p(X, \gamma_s)$, such that 
$\|\pst \|_{{\mathcal L}(L^p(X, \gamma_t ); L^p(X, \gamma_s))} \leq 1$. Taking  $f\equiv1$,  $\pst f \equiv 1 $ so that  
$\|\pst\|_{\mathcal{L}(L^p(X, \gamma_t ); L^p(X, \gamma_s)}=1$. 
\end{proof}

\begin{oss}In general the spaces $L^p(X,\g_t)$ and $L^p(X,\g_s)$ are different if $t \neq s$, and the classical theory of evolution operators in fixed Banach spaces cannot be used.
\end{oss}

\begin{thm}\label{HYPER}
Assume that Hypothesis \ref{4} holds true. Moreover we assume that $U(t,s)\ac_s\subseteq\ac_t$ for every $(s,t)\in\Delta$ and that there exist $C,\eta>0$ and $\displaystyle{\al\in \left[0,\frac{1}{2}\right)}$ such that
\begin{equation*}
\norm{U(t,s)_{|_{\ac_s}}}_{\op(\ac_s;\ac_t)}\leq C\, \frac{e^{-\eta(t-s)}}{(t-s)^\al},\qquad\forall\; (s,t)\in\Delta.
\end{equation*}
Then, for every $(s,t)\in\Delta$, $q\in (1,+\infty)$ and $p\leq (q-1)e^{\frac{t-s}{2\kappa}}+1$ we have
\begin{equation}\label{iper}
\norm{\pst \ph}_{L^p(X,\, \g_s)}\leq \norm{\ph}_{L^q(X,\, \g_t)},\quad \ph\in L^q(X,\, \g_t).
\end{equation}
where $\kappa$ is given by \eqref{logsobcost}.
\end{thm}
\begin{proof}
Let $(s,t)\in\Delta$ and let $r\in [s,t]$. We prove first that \eqref{iper} holds for $\ph\in \mathfrak{B}^2_t(X)$ such that $\displaystyle{\inf_{x\in X}\ph(x)>\ep}$ for some $\ep>0$. 
We consider the mapping $\psi:[s,t]\times X\rightarrow\R$ given by
\begin{align}\label{mappag2}
\psi(r,x):=\left(P_{r,t}\ph\right)(x).
\end{align}
By Lemma \ref{st-bohr}, $\psi(r,\cdot)\in\mathfrak{B}^2_{r}(X)$ for every $r\in [s,t]$  and $$\inf_{r\in [s,t]}\inf_{x\in X}\psi(r,x)>\ep$$ and thanks to \eqref{derivata-sx} we get
\begin{align}
\dfrac{\partial \psi}{\partial r}(r,x)=-L(r)\psi(r,\cdot)(x),\quad r\in [s,t],\ x\in X.
\end{align}
 Now we define the functions $G:[s,t]\rightarrow \R$ and $H:[s,t]\rightarrow \R$ by
\begin{equation}
G(r):=\int_X \psi(r,x)^{p(r)}\, \g_r(dx),\qquad H(r)=G(r)^{\frac{1}{p(r)}},
\end{equation}
where
\begin{equation}\label{pr}
p(r):=(q-1)e^{(2\kappa)^{-1}(t-r)}+1.
\end{equation}
By Lemma \ref{lemma31} we have 
\begin{align}
G'(r)&=\int_X \psi(r,x)^{p(r)}\log \psi(r,x)\,p'(r)\, \g_r(dx)-\int_X \psi(r,x)^{p(r)-1}p(r)L(r)\psi(r,\cdot)(x)\, \g_r(dx)\nonumber\\
&+\int_X L(r)\psi(r,\cdot)^{p(r)}(x)\, \g_r(dx).\label{I1}
\end{align}
For every $r\in [s,t]$, $p\geq 1$ and $\Phi\in\mathcal{F}_{r}C_b^2(X)$, by \eqref{ftgeneratore} we have
\begin{equation}\label{PotenzaL}
L(r)\Phi^p=\frac{1}{2}p(p-1)\Phi^{p-2}\norm{\Qp{r}\nabla \Phi}^2+p\Phi^{p-1}L(r)\Phi,
\end{equation}
and applying \eqref{PotenzaL} to \eqref{I1}, we get
\begin{align}\label{I2}
G'(r)&=\int_X \psi(r,x)^{p(r)}\log \psi(r,x)\,p'(r)\, \g_r(dx)\notag\\
&+\frac{1}{2}p(r)(p(r)-1)\int_X\psi(r,x)^{p-2}\norm{\Qp{r}\nabla \psi(r,x)}^2\, \g_r(dx).
\end{align}
By \eqref{I2} we obtain
\begin{align*}
(\log H(r))'&=\frac{1}{p(r)G(r)}G'(r)-\frac{p'(r)}{p(r)^2}\log G(r)\\
&=\frac{p'(r)}{p(r)\int_X\psi(r,x)^{p(r)}\, \g_r(dx)}\int_X \psi(r,x)^{p(r)}\log \psi(r,x)\, \g_r(dx)\notag\\
&+\frac{p(r)-1}{2\int_X\psi(r,x)^{p(r)}\, \g_r(dx)}\int_X\psi(r,x)^{p-2}\norm{\Qp{r}\nabla \psi(r,x)}^2\, \g_r(dx)\\
&-\frac{p'(r)}{p(r)^2}\log\left(\int_X\psi(r,x)^{p(r)}\, \g_r(dx)\right)\\
&=\frac{p'(r)}{p(r)^2G(r)}\left[\int_X \psi(r,x)^{p(r)}\log \psi(r,x)^{p(r)}\, \g_r(dx)-m_r(\psi(r,\cdot)^p)\log m_r(\psi(r,\cdot)^p)\right]\\
&+\frac{p(r)-1}{2G(r)}\int_X\psi(r,x)^{p-2}\norm{\Qp{r}\nabla \psi(r,x)}^2\, \g_r(dx).
\end{align*}
Taking into account that $p'(r)< 0$, by \eqref{logsobolev} we get
\begin{equation}
(\log H(r))'\geq \frac{1}{2G(r)}\left(2p'(r)\kappa+p(r)-1)\right)\int_X\psi(r,x)^{p-2}\norm{\Qp{r}\nabla \psi(r,x)}^2\, \g_r(dx)=0.
\end{equation}
Since $p(r)$ is given by \eqref{pr}, $\log H(r)$ is a non decreasing function so also $H(r)$ is a non decreasing function. Hence \eqref{iper} holds true for every $\ph\in\mathfrak{B}^2_t(X)$ with positive infimum.  

We obtain \eqref{iper} for every $\ph\in \mathfrak{B}^2_t(X)$ applying \eqref{iper} to the standard approximation $\displaystyle{\ph_n=\sqrt{\ph^2+\frac{1}{n^2}}}$ and letting $n\rightarrow +\infty$ in  \eqref{iper}. We stress that, fixed $n\in\Nat$, the function $\varphi_n$ belongs to $\mathfrak{B}^2_t(X)$ since the function $h(x)=\sqrt{x^2+\frac{1}{n}}$ has continuous and bounded derivatives of every order. 
Since $\mathcal{E}_t(X)\subset\mathfrak{B}^2_t(X)$ \eqref{iper} holds for all $\ph\in L^q(X,\g_t)$ by Remark \ref{densità}. 
%The statement follows by the same approximation arguments used at end of the proof of Theorem \ref{LOG}
\end{proof}

\section{Examples}\label{Examples}
In this section we give three genuinely non autonomous examples.

\subsection{A non autonomous parabolic problem}
Let $d\in\Nat$ and let $\os\subseteq\R^d$ be a bounded open set with smooth boundary. We consider the evolution operator $\{U(t,s)\}_{(s,t)\in\overline{\Delta}}$ in $X:=L^2(\os)$ associated to an evolution equation of parabolic type,

\begin{align}\label{parabolico}
\begin{cases}
u_t(t,x)=\oa(t)u(t,\cdot)(x),\ \ (t,x)\in(s,+\infty)\times\os,\\
\mathcal{B}(t)u(t,\cdot)(x)=0,\ \ (t,x)\in(s,+\infty)\times\de\os.\\
\end{cases}
\end{align}
The differential operators $\oa(r)$ are defined by
\begin{equation}\label{operatoreat}
\oa(r)\ph(x)=\sum_{i,j=1}^d D_i\left(a_{ij}(r,x)D_{i}\ph(x)\right)+\sum_{i=1}^d a_{i}(r,x)D_{i}\ph(x)+a_0(r,x)\ph(x),\quad r\in \R,\ x\in\os
\end{equation}
and the family of the boundary operators $\{\mathcal{B}(r)\}_{r\in \R}$ is either of Dirichlet or Robin type, namely
\begin{align}\label{boundary}
\mathcal{B}(r)u=
\begin{cases}
\begin{aligned}
&u \ \ \quad \quad\quad \quad \quad\quad\quad \quad \quad \quad\quad \quad \quad \quad \mbox{(Dirichlet)},  \\[1ex]
%&\displaystyle{\sum_{i,j=1}^da_{ij}(x,r)D_i u\,\nu_j}\quad \quad\quad\quad \quad \quad\, \ \ \mbox{(Neumann)},\\[1ex] 
&\displaystyle{\sum_{i,j=1}^da_{ij}(x,r)D_i u\,\nu_j}+b_0(x,r)u\quad \quad \mbox{(Robin)},
\end{aligned}
\end{cases}
\end{align}
 where $\nu=(\nu_1,...,\nu_d)$ is the unit outer normal vector at the boundary of $\Omega$. 
 
 We make the following assumptions.

\begin{ip}\label{A}
We assume $a_{ij}=a_{ji}$ and for some $\displaystyle{\rho\in\left(\frac{1}{2},1\right)}$, $a_{ij}\in C_b^{\rho,2}\left(\R\times\overline{\os}\right)$, $b_0\in C_b^{\rho,1}\left(\R\times\overline{\os}\right)$, $a_{i},a_0\in C_b^{\rho,0}\left(\R\times\overline{\os}\right)$. Moreover, we assume that there exist $\nu, \omega, \be_0>0$ such that 
\begin{align}
&\sum_{i,j=1}^da_{ij}(r,x)\xi_i\xi_j\geq \nu\abs{\xi}^2,\ \ r\in\R,\ x\in\os,\  \xi\in\R^d,\\
&\sup_{(r,x)\in\R\times\os}a_0(r,x)\leq -\om,\\
&\inf_{(r,x)\in\R\times\os}b_0(r,x)\geq\be_0, \\
& \delta_0-\omega<0,
\end{align}
where $\displaystyle{\delta_0=\frac{1}{\nu}\left(\sum_{i=1}^d\norm{a_i}^2_\infty\right)^{\frac{1}{2}}}$.
\end{ip}
For every $r\in\R$ we denote by $A(r)$ the realization in $L^2(\os)$ of $\oa(r)$ with one of the boundary conditions \ref{boundary}. In \cite[Ex.\,2.8, Ex.\,2.9, Rmk 3.19 and Ex. 4.9]{Schnaubelt2004} it is proven that  the family $\{A(r)\}_{r\in\R}$ satisfies the assumptions of \cite{MR945820, MR934508},  so there exists an evolution operator  $\{U(t,s)\}_{s\leq t}$ on $X$ such that Hypothesis \ref{2bis} holds true.

\begin{prop}\label{zeta=omega}
 Under Hypothesis \ref{A}, \eqref{omega} holds with $\zeta=\omega-\delta_0$.
\end{prop}
\begin{proof}
We consider the family of operators $\left\{-\widetilde{A}(r)\right\}_{r\in\R}$ defined by
\begin{equation}
-\widetilde{A}(r)=-A(r)+\om\id_X,\ \ r\in\R.
\end{equation} 
Since $\left\{A(r)\right\}_{r\in\R}$ verifies Hypothesis \ref{2bis} then $\left\{\widetilde{A}(r)\right\}_{r\in\R}$ verifies Hypothesis \ref{2bis} and it is associated to the evolution operator $\{\widetilde{U}(t,s)\}_{s\leq t}$, given by
\begin{equation}\label{utstilde}
\widetilde{U}(t,s)=e^{\om(t-s)}\uts,\quad s\leq t.
\end{equation}
Moreover $\left\{-\widetilde{A}(r)\right\}_{r\in\R}$ satisfies all hypotheses of \cite{MR1780769} and  by \cite[Thm\, 5.1]{MR1780769} we have
\begin{align}
\norm{\widetilde{U}(t,s)}_{\op(X)}\leq e^{\delta_0(t-s)},\quad s\leq t,
\end{align}
and
\begin{align}\label{santograal}
\norm{U(t,s)}_{\op(X)}=e^{-\omega(t-s)}\norm{\widetilde{U}(t,s)}_{\op(X)}\leq e^{-(\omega-\delta_0)(t-s)},\quad s\leq t.
\end{align}
\end{proof}

\begin{prop}\label{traccia-finita}
Assume that Hypothesis \ref{A} holds true. Let $\{B(r)\}_{r\in\R}$ be a family of operators such that
\begin{enumerate}
\item for every $r\in\R$, $B(r)\in\mathcal{L}(L^2(\os),L^q(\os))$ for some $q\in [2,+\infty)\cap (d,+\infty)$ and
 \[
\sup_{r\in\R}\norm{B(r)}_{\op(L^2(\os);L^q(\os))}<+\infty;
\]  
\item for every $\ph\in L^2(\os)$ the mapping $r\in\R\longmapsto B(r)\ph\in L^q(\os)$ is continuous.
\end{enumerate} 
Then
\begin{equation}\label{stimafinale}
\Tr{Q(t,s)}\leq C^2\abs{\os}M_{q'}^{\frac{2}{q'}}\int_s^t\frac{e^{-\om(t-r)}}{(t-r)^{\frac{d}{q}}}\,dr.
\end{equation}
 and the operator $Q(t,s)$  has finite trace for all $s<t$. Moreover there exists a unique evolution system of measures $\{\gamma_t\}_{t\in \R}$ for $\pst$ in $\R$ given by \eqref{gammat}.
\end{prop}
\begin{proof}
We adapt to our setting the arguments of \cite[Lemma\,4.3]{cerlun}.

By \cite{MR1780769} $\widetilde{U}(t,s)$ defined in \eqref{utstilde} may be exended to the whole $L^1(\os)$, and the extension (still denoted by $\widetilde{U}(t,s)$) belongs to $\op\left(L^1(\os); L^\infty(\os)\right)$. Moreover $\widetilde{U}(t,s)$ is represented by
\begin{equation}
\widetilde{U}(t,s)\ph(x)=\int_\os k(x,y,t,r)\ph(y)\,dy,\quad \ph\in L^1(\os),\ s<t,
\end{equation}
where $k(\cdot,\cdot,t,s)$ belongs to $L^\infty\left(\os\times\os\right)$ and by \cite[thm.\,6.1]{MR1780769} there exist $M, m>0$ such that
\begin{equation}\label{stimakergauss1}
\abs{k(x,y,t,r)}\leq\frac{M}{(t-s)^{\frac{d}{2}}}\,e^{-\frac{\abs{x-y}^2}{m(t-s)}},\quad x,y\in\os,\ s<t,
\end{equation}
%By \cite[Thm. 5.1]{MR1780769}, there exists $M>0$ such that 
%\begin{equation}
%\norm{\widetilde{U}(t,s)}_{\op(X)}\leq M, \quad s<t,
%\end{equation}
and recalling \eqref{utstilde}, we get
\begin{align}
\label{kergauss}
&U(t,s)\ph(x)=e^{-\om (t-s)}\int_\os k(x,y,t,r)\ph(y)\,dy,\quad \ph\in L^1(\os),\ \ s<t,\ x\in\os.
%&\norm{U(t,s)}_{\op(X)}\leq M\, e^{-\om (t-s)}\quad s<t. \label{stimakergauss2}
\end{align} 
Let $\{e_k\}_{k\in\Nat}$ be a Hilbert basis of $X$. Then
\begin{equation}
\Tr{Q(t,s)}=\int_s^t\sum_{k=1}^\infty\norm{B(r)^\star U(t,r)^\star e_k}^2_{L^2(\os)}\,dr,\quad s<t.
\end{equation}
By the representation formula \eqref{kergauss}, we get
\begin{equation}
(U(t,r)^\star e_k)(y)=e^{-\om (t-r)}\int_\os k(x,y,t,r)e_k(x)\,dx,\quad\mbox{a.e.}\ y\in\os,
\end{equation}
and then 
\begin{equation}
(B(r)^\star U(t,r)^\star e_k)(y)=e^{-\om (t-r)}\int_\os \left(B(r)^\star k(x,\cdot,t,r)\right)(y)\,e_k(x)\,dx,\quad\mbox{a.e.}\ y\in\os.
\end{equation}
We obtain
\begin{align}\label{contitraccia}
\Tr{Q(t,s)}&=\int_s^t e^{-\om (t-r)}\int_\os\sum_{k=1}^\infty\left(\int_\os \left(B(r)^\star k(x,\cdot,t,r)\right)(y)\,e_k(x)\,dx\right)^2dydr \nonumber\\
&=\int_s^t e^{-\om (t-r)}\int_\os\int_\os \left(\left(B(r)^\star k(x,\cdot,t,r)\right)(y)\right)^2\,dxdydr\nonumber\\
&=\int_s^t e^{-\om (t-r)}\int_\os\int_\os \left(\left(B(r)^\star k(x,\cdot,t,r)\right)(y)\right)^2\,dydxdr.
\end{align}
Since $B(r)\in\mathcal{L}(L^2(\os);L^q(\os))$ then $B(r)^\star\in\mathcal{L}(L^{q'}(\os);L^2(\os))$ and  there exists $C>0$ such that for every $x\in\os$ we have 
\begin{align}
\norm{B(r)^\star k(x,\cdot,t,r)}_{L^2(\os)}&\leq \norm{B(r)^\star}_{\op(L^{q'}(\os);L^2(\os))}\norm{k(x,\cdot,t,r)}_{L^{q'}(\os)}\nonumber\\
&\leq C \norm{k(x,\cdot,t,r)}_{L^{q'}(\os)}.
\end{align}
By \eqref{stimakergauss1}, for every $p>1$ there exists $M_p>0$ independent of $x$ such that
\begin{align}
\norm{k(x,\cdot,t,r)}^p_{L^{p}(\os)}\leq \frac{M}{(t-r)^{\frac{dp}{2}}}\int_{\R^d}e^{-\frac{p}{m}\frac{\abs{x-y}}{t-r}}\,dy=\frac{M_p}{(t-r)^{\frac{d(p-1)}{2}}}.
\end{align}
Choosing $p=q'$ we obtain
\begin{align}\label{prestimafinale}
\norm{B(r)^\star k(x,\cdot,t,r)}^2_{L^2(\os)}\leq\frac{C^2M_{q'}^{\frac{2}{q'}}}{(t-r)^{\frac{d}{q}}}.
\end{align}
Combining \eqref{contitraccia} and \eqref{prestimafinale} we obtain
\begin{equation}
\Tr{Q(t,s)}\leq C^2\abs{\os}M_{q'}^{\frac{2}{q'}}\int_s^t\frac{e^{-\om(t-r)}}{(t-r)^{\frac{d}{q}}}\,dr.
\end{equation}
where $\abs{\os}$ is the $d$-dimensional Lebesgue measure of $\os$.

Since $q>d$, \eqref{stimafinale} implies that the trace of $Q(t,s)$ is finite for every $s<t$ and 
\[
\sup_{s<t}\left[\Tr{Q(t,s)}\right]<+\infty.
\] 
By Corollary \ref{unicità}, Remark \ref{omeganegativo} and \eqref{santograal}, there exists a unique evolution system of measures for $\pst$ in $\R$ and it is given by \eqref{gammat}.
\end{proof}

We just have to give sufficient conditions guaranteeing that \eqref{cond-logsob2} holds true. To this aim we need to recall some preliminary results.
By \cite[Thm.\,2.3]{MR945820}, $\uts\in\op(X;D(A(t)))$ and there exists $C_1>0$ such that for every $0<t-s\leq 1$ we have
\begin{align}
&\norm{\uts}_{\op(X)}+\norm{A(t)\uts}_{\mathcal{L}(X)}=\norm{\uts}_{\op(X;D(A(t)))}\leq\frac{C_1}{t-s},\label{SAT1.1}\\
&\norm{\uts}_{\op(D(A(s));D(A(t)))}\leq C_1.\label{SAT2.1}
\end{align} 
By Proposition \ref{zeta=omega} and \eqref{SAT1.1}, for every $t-s>1$ we have
\begin{align}
\norm{\uts}_{\op(D(A(s));D(A(t)))}&\leq\norm{\uts}_{\op(X;D(A(t)))}\nonumber\\
&=\norm{U(t,t-1)U(t-1,s)}_{\op(X)}+\norm{A(t)U(t,t-1)U(t-1,s)}_{\mathcal{L}(X)}\nonumber\\
&\leq \left(\norm{U(t,t-1)}_{\op(X)}+\norm{A(t)U(t,t-1)}_{\mathcal{L}(X)}\right)\norm{U(t-1,s)}_{\op(X)}\nonumber\\
&\leq C_1e^{\omega-\delta_0}e^{-(\omega-\delta_0)(t-s)}\label{SAT3.1}
\end{align} 
so by Proposition \ref{zeta=omega} and combining \eqref{SAT1.1}, \eqref{SAT2.1} and \eqref{SAT3.1} there exists $C>0$ such that for every $s<t$ we have
\begin{align}
&\norm{\uts}_{\mathcal{L}(X)}\leq C\,e^{-(\omega-\delta_0)(t-s)},\label{SAT0}\\
&\norm{\uts}_{\op(X;D(A(t)))}\leq C\,\max\left\{1,(t-s)^{-1}\right\}\,e^{-(\omega-\delta_0)(t-s)},\label{SAT1}\\
&\norm{\uts}_{\op(D(A(s));D(A(t)))}\leq C\,e^{-(\omega-\delta_0)(t-s)}.\label{SAT2}
\end{align} 
By \eqref{SAT0} and \eqref{SAT2} for every $0<\theta<1$ and for every $s<t$ we have
\begin{equation}\label{SAT3}
\norm{\uts}_{\op\left((X,D(A(s))_{\theta,2};(X,D(A(t))_{\theta,2})\right)}\leq C\,e^{-(\omega-\delta_0)(t-s)},
\end{equation}
where, for every $r\in\R$, $(X,D(A(r))_{\theta,2}$ is the standard real interpolation space.
Moreover by \eqref{SAT0} and \eqref{SAT1} for every $0<\theta<1$ and for every $s<t$ we have
\begin{equation}\label{SAT4}
\norm{\uts}_{\op\left(X;(X,D(A(t))_{\theta,2})\right)}\leq C\,\max\{1,(t-s)^{-\theta}\}\,e^{-(\omega-\delta_0)(t-s)}.
\end{equation}
Combining \eqref{SAT3} and \eqref{SAT4} for every $0<\sigma<1$ and $s<t$ we get
\begin{equation}\label{SAT5pre}
\norm{\uts}_{\op\left((X,(X,D(A(s))_{\theta,2})_{\sigma,2};(X,D(A(t))_{\theta,2})\right)}\leq C\,\max\left\{1,(t-s)^{-\theta(1-\sigma)}\right\}\,e^{-(\omega-\delta_0)(t-s)}.
\end{equation}
Recalling that by reiteration $(X,(X,D(A(s))_{\theta,2})_{\sigma,2}=(X,D(A(s))_{\theta\sigma,2}$, for every $\theta\in (0,1)$, $\rho\in (0,\theta]$ and $s<t$ we get
\begin{equation}\label{SAT5}
\norm{\uts}_{\op\left((X,D(A(s))_{\rho,2};(X,D(A(t))_{\theta,2})\right)}\leq C\,\max\{1,(t-s)^{-(\theta-\rho))}\}\,e^{-(\omega-\delta_0)(t-s)},
\end{equation}
 where we have chosen $\displaystyle{\sigma=\frac{\rho}{\theta}}$ in \eqref{SAT5pre}. 
Now we recall the characterization of $D(A(r))$ for every $r\in\R$.
 \begin{itemize}
 \item \textbf{Robin boundary condition} In this case for every $r\in\R$ we have
\begin{align} 
D(A(r))&=\left\{u\in H^2(\os):\ \mathcal{B}(r)u=0\right\}.
\end{align} 
Moreover by e.g. \cite[Thm.\,3.5, Thm.\,4.15]{MR1115176} for every $r\in\R$ and $0<\g<1$, we have
\begin{align}\label{guidetti0}
(X,D(A(r))_{\g,2}&=\begin{cases}
 H^{2\g}(\os) &\mbox{if}\ 0<\g<\frac{3}{4}\\[1ex]% \phantom{a}\\
 \left\{u\in H^{\frac{3}{2}}(\os)\ |\ \tilde{\mathcal{B}}(r)u\in\mathring{H}^{\frac{1}{2}}(\os)\right\} &\mbox{if}\ \g=\frac{3}{4} \\[1ex]
\left\{u\in H^{2\g}(\os)\ |\ \mathcal{B}(r)u=0\right\}\quad &\mbox{if}\ \frac{3}{4}<\g<1
\end{cases},
\end{align}
where
\begin{align}
\tilde{\mathcal{B}}(r)u=
%\begin{cases}
%\begin{aligned}
%&\displaystyle{\sum_{i,j=1}^da_{ij}(x,r)D_i u\,\tilde{\nu_j}}\quad \quad\quad\quad \quad \quad\, \ \ \mbox{(Neumann)}\\[1ex] 
\displaystyle{\sum_{i,j=1}^da_{ij}(x,r)D_i u\,\tilde{\nu_j}}+b_0(x,r)u,
%\end{aligned}
%{cases},
\end{align} 
$\tilde{\nu}$ is a smooth enough extension of $\nu$ to $\overline{\os}$ and $\mathring{H}^{\frac{1}{2}}(\os)$ consists on all the elements $\ph\in H^{\frac{1}{2}}(\os)$ whose null extension outside $\overline{\os}$ belongs to $H^{\frac{1}{2}}(\R^d)$.

\item \textbf{Dirichlet boundary condition} In this case for every $r\in\R$ we have
\begin{align} 
D(A(r))&=H^2(\os)\cap H^1_0(\os).
\end{align} 
Moreover by \cite[Thm.\,3.5, Thm.\,4.15]{MR1115176} for every $r\in\R$ and $0<\g<1$, we have
\begin{align}\label{guidetti2}
\left(L^2(\os),H^2(\os)\cap H^1_0(\os)\right)_{\g,2}=\begin{cases}
 H^{2\g}(\os) &\mbox{if}\ 0<\g<\frac{1}{4}\\% \phantom{a}\\
 \mathring{H}^{\frac{1}{2}}(\os) &\mbox{if}\ \g=\frac{1}{4} \\
\left\{u\in H^{2\g}(\os)\ |\ u_{|_{\de\os}}=0\right\}\quad &\mbox{if}\ \frac{1}{4}<\g< 1
\end{cases}.
\end{align}
\end{itemize}
 
Now we present two explicit examples of $\{B(r)\}_{r\in\R}$ where all the hypotheses of Theorems \ref{LOG} and \ref{HYPER} are verified.

\begin{thm}\label{thm7.4}
Assume the following conditions hold true.
\begin{enumerate}
\item $X:=L^2(\os)$ where $\os\subseteq\R^d$ is a bounded open set with smooth boundary and $d=1,2,3,4,5$. 
\item The operators  $\oa(r)$ given by \eqref{operatoreat} verify Hypothesis \ref{A}.
\item The realization $A(r)$ in $L^2(\os)$ of $\oa(r)$ with one of the boundary conditions \ref{boundary} is a negative operator for every $r\in\R$.
\item For every $r\in\R$ we have 
\[
B(r)=(-A(r))^{-\gamma},\quad r\in\R,
\]
for some $\gamma\geq 0$.
\end{enumerate}
Then Hypothesis \ref{4} holds true in the following cases
\begin{align}\label{sceltagamma2}
\begin{cases}
0\leq\g<1\ \ \ \ \mbox{if}\ d=1\\[1ex]
0<\g<1\ \ \ \ \mbox{if}\ d=2\\[1ex]
 \frac{1}{4}<\g<1\ \ \ \, \mbox{if}\ d=3\\[1ex]
\frac{1}{2}<\g<1\ \ \ \ \mbox{if}\ d=4\\[1ex]
\frac{3}{4}<\g<1\ \ \ \ \mbox{if}\ d=5
\end{cases}.
\end{align}
Moreover for every $s<t$ we have
\begin{equation}\label{stima-logsob1}
\norm{U(t,s)}_{\op({\rm H}_s,{\rm H}_t)}\leq C e^{-(\omega-\delta_0)(t-s)},
\end{equation}
where $C$, $\omega$ and $\delta_0$ are the constants appearing in \eqref{SAT3}.
\end{thm}
\begin{oss} If $a_i\equiv 0$ for all $i=1,..,d$ hypothesis \textit{3} of Theorem \ref{thm7.4} is satisfied. If not, by standard arguments, one can find sufficient conditions on the coefficients $a_i$ such that hypothesis \textit{3} holds.
\end{oss}
\begin{proof}
By \cite[Thm.\,4.36]{MR3753604} we get
\begin{equation}\label{domini-potenze}
{\rm H}_r:=\Qp{r}(X)=D((-A(r))^\gamma)=(X,D(A(r)))_{\gamma,2},
\end{equation}
It follows that for every $q\geq 2$, the embedding of $D\left((-A(r))^\gamma\right)$ in $L^q(\os)$ is continuous for $\g\geq 2d \left(\frac{1}{2}-\frac{1}{q}\right)$. Hence for such choices of $\g$, $(-A(r))^{-\g}\in\op(L^2(\os);L^q(\os))$. 

So by Proposition \ref{traccia-finita} and recalling \eqref{guidetti0} \eqref{guidetti2}, for every $s<t$ the operator $Q(t,s)$ given by \eqref{qtscov} has finite trace in all the cases \eqref{sceltagamma2}. Hypothesis \ref{4} holds true in view of  Proposition \ref{zeta=omega}. Finally \eqref{stima-logsob1} holds true by \eqref{SAT3} and \eqref{domini-potenze} with $\theta=\g$.
\end{proof}

\begin{thm}
Assume the following conditions hold true.
\begin{enumerate}
\item $X:=L^2(\os)$ where $\os\subseteq\R^d$ is a bounded open set with smooth boundary and $d=1,2,3$. 
\item The operators  $\oa(r)$ given by \eqref{operatoreat} verify Hypothesis \ref{A}. 
\item For every $r\in\R$ we have 
\[
B(r)=(-\Delta)^{-\gamma(r)},\quad r\in\R,
\]
where $\Delta$ is the realization of the Laplacian operator in $L^2(\os)$ with Dirichlet (Robin) boundary conditions, $\gamma:\R\rightarrow [0,\alpha]$ is a non-decreasing continuous function and $\displaystyle{0<\alpha<\frac{1}{2}}$. 
\end{enumerate}
Then Hypothesis \ref{4} holds true in the following cases
\begin{align}\label{sceltagamma2.2}
\begin{cases}
\displaystyle{\inf_{r\in\R}\gamma(r)\geq 0}\ \ \ \ \mbox{if}\ d=1\\[1ex]
\displaystyle{\inf_{r\in\R}\gamma(r)> 0}\ \ \ \ \mbox{if}\ d=2\\[1ex]
\displaystyle{\inf_{r\in\R}\gamma(r)>\frac{1}{4}}\ \ \ \, \mbox{if}\ d=3
\end{cases}.
\end{align}
Moreover for every $s<t$ we have
\begin{equation}\label{stima-logsob2}
\norm{U(t,s)}_{\op({\rm H}_s,{\rm H}_t)}\leq C\max\{1,(t-s)^{-\alpha}\}e^{-(\om-\delta_0)(t-s)},
\end{equation}
where $C,\omega, \delta_0$ are the constants appearing in \eqref{SAT5}.
\end{thm}
\begin{proof}
We prove the statement in the case where $\{A(r)\}_{r\in\R}$ is the realization in $L^2(\os)$ of $\oa(r)$ with Robin boundary conditions, the case with Dirichlet can be treated in the same way.
Since $\displaystyle{\sup_{r\in\R}\gamma(r)<\frac{1}{2}}$ by \eqref{guidetti0} and \eqref{guidetti2} for every $r\in\R$ we have
\begin{equation}\label{domini-potenze2}
{\rm H}_r:=\Qp{r}(X)=D((-\Delta)^{\gamma(r)})=(X,D(A(r)))_{\gamma,2}=H^{2\gamma}(\os).
\end{equation}
It follows that for every $q\geq 2$, the embedding of $D\left((-A(r))^\gamma\right)$ in $L^q(\os)$ is continuous for $\g\geq 2d \left(\frac{1}{2}-\frac{1}{q}\right)$. Hence for such choices of $\g$, $(-A(r))^{-\g}\in\op(L^2(\os);L^q(\os))$. 

So by Proposition \ref{traccia-finita}, for every $s<t$ the operator $Q(t,s)$ given by \eqref{qtscov} has finite trace in all cases \eqref{sceltagamma2.2}. Hypothesis \ref{4} holds true in view of Proposition \ref{zeta=omega} and for every $s<t$, by \eqref{SAT5}(with $\theta=\gamma(t)$ and $\rho=\gamma(s)$) and \eqref{domini-potenze2} we have
\begin{equation*}
\norm{\uts}_{\op\left({\rm H}_s,{\rm H}_t\right)}\leq C\max\{1,(t-s)^{-(\gamma(t)-\gamma(s))}\}e^{-(\omega-\delta_0)(t-s)},
\end{equation*}
so recalling that $\gamma:\R\rightarrow [0,\alpha]$ is a non-decreasing continuous function and $\displaystyle{0<\alpha<\frac{1}{2}}$ we obtain \eqref{stima-logsob2}.
 \end{proof}
 
 \subsection{Diagonal operators}
Let $(X,\norm{\cdot}_X,\ix{\cdot}{\cdot})$ be a separable Hilbert space. Let $t\in\R$ and let $A(t)$, $B(t)$ be self-adjoint operators in diagonal form with respect to the same Hilbert basis $\{e_k:\ k\in\Nat\}$, namely $$A(t)e_k=a_k(t)e_k,\ \ B(t)e_k=b_k(t)e_k\ \ t\in\R,\ \ k\in\Nat,$$ with continuous coefficients $a_k$, $b_k$. We set $\displaystyle{\la_k=\sup_{t\in\R}a_k(t)}$ and we assume that there exists $\la_0\in\R$ such that $\la_k\leq\la_0,\ \ \forall\ k\in\Nat.$

In this setting  the operator $\uts$ defined by $$U(t,s)e_k=\exp\biggl(\int_s^ta_k(\tau)\,d\tau\biggr)e_k,\ \ (s,t)\in\Delta,\ \ k\in\Nat,$$ is the strongly continuous evolution operator associated to the family $\{A(t)\}_{t\in\R}$.
Moreover we assume that there exists $K>0$ such that $$\abs{b_k(t)}\leq K,\ \ t\in\R,\ k\in\Nat.$$
Hence $B(t)\in\op(X)$ for all $t\in\R$, the function $B:\R\longmapsto\op(X)$ is continuous and $$\sup_{t\in\R}\norm{B(t)}_{\op(X)}\leq K.$$
The operators $Q(t,s)$ are given by
\begin{equation}
Q(t,s)e_k=\int_s^t\exp\biggl(2\int_\si^ta_k(\tau)\,d\tau\biggr)(b_k(\si))^2\,d\si\,e_k=:q_k(t,s)e_k,\ \ (s,t)\in\Delta,\ k\in\Nat. \nonumber
\end{equation}
Hypothesis \ref{1} is fulfilled if 
\begin{equation}\label{qk1}
\sum_{k=0}^\infty q_k(t,s)<+\infty,\ \ (s,t)\in\Delta.
\end{equation}
We give now a sufficient condition for \eqref{qk1} to hold.
We assume that $\la_k$ is eventually nonzero (say for $k\geq k_0$). Given $(s,t)\in\Delta$, we have
\begin{align}\label{qk2}
&\abs{\int_s^t\exp\biggl(2\int_\si^ta_k(\tau)\,d\tau\biggr)(b_k(\si))^2\,d\si}\leq \norm{b_k}^2_{\infty}\abs{\int_s^t\exp\bigl(2\la_k(t-\si)\bigr)\,d\si}\nonumber \\
&=\frac{\norm{b_k}^2_{\infty}}{2\abs{\la_k}}\abs{1-\exp(2\la_k(t-s))}\leq\frac{\norm{b_k}^2_{\infty}}{2\abs{\la_k}}\bigl(1+\exp(2\la_0(t-s))\bigr).
\end{align}
Hence \eqref{qk1} holds if we require 
\begin{equation}
\sum_{k=k_0}^\infty\frac{\norm{b_k}^2_{\infty}}{\abs{\la_k}}<+\infty.
\end{equation}
We note that $D(A(t))=D(A(t)^\star)=X$ for all $t\in\R$ and Hypothesis \ref{2bis} is easily checked.

In order to have existence of an evolution system of measures for $\pst$ in $\R$, the hypotheses of Theorem \ref{esistenza} are satisfied if we require $\la_0\leq 0$ in \eqref{qk2}. Moreover if we require $\la_0<0$, \eqref{omega} holds with $\zeta<0$ and by Theorem \ref{unicità} there exists a unique system of measures for $\pst$ in $\R$.

Now we investigate when \eqref{cond-logsob2} holds.
We observe first that for $y=\Qp{s}x$ with $x\in H_s$, we have 
\begin{align}
&\norm{U(t,s)y}^2_{H_t}=\norm{\Qm{t}U(t,s)\Qp{s}x}_X^2=\sum_{\substack{k\in\Nat}}\biggl(\frac{\abs{b_k(s)}}{\abs{b_k(t)}}\exp\biggl(\int_s^ta_k(\tau)\biggr)\iprod{x}{e_k}\biggr)^2\nonumber\\
&=\sum_{\substack{k\in\Nat}}\biggl(\frac{\abs{b_k(s)}}{\abs{b_k(t)}}\exp\biggl(\int_s^ta_k(\tau)\biggr)\frac{\iprod{y}{e_k}}{\abs{b_k(s)}}\biggr)^2=\sum_{\substack{k\in\Nat}}\biggl(\frac{1}{\abs{b_k(t)}}\exp\biggl(\int_s^ta_k(\tau)\iprod{y}{e_k}_{H_s}\biggr)\biggr)^2. \nonumber
\end{align}

We assume that there exist $L>0$ such that $\abs{b_k(t)}\geq L$ for all $k\in\Nat$ and $t\in\R$.
Hence for any $k\in\Nat$, we have
\begin{align} 
\frac{1}{b_k^2(t)}\exp\biggl(\int_s^t2a_k(\tau)\,d\tau\biggr)\leq \frac{1}{L^2}e^{2\la_0 (t-s)},
\end{align}
and
\begin{equation}
\norm{\uts}_{\op(H_s,H_t)}\leq \frac{1}{L}e^{\la_0 (t-s)},\ \ (s,t)\in\Delta.
\end{equation}
Since $\la_0<0$, hypotheses of Theorems \ref{LOG} and \ref{HYPER} are satisfied.

As an explicit example we can choose $\displaystyle{a_k(t)=-\frac{k^2+c_1}{t^{2k}+1}}$, $c_1>0$ and $b_k(t)=\sin(kt)+c_2$, $c_2>1$ for all $t\in\R$.  

\begin{oss}\label{nonunico} We assume now that $\displaystyle{\max_{t\in\R}a_1(t)=0}$, $a_1\in L^1(\R)$ and there exists $c>0$ such that $\la_k<-c$ for $k\geq 2$. In this case $\la_0=0$ and we can show that there exist at least two different evolution system of measures uniformly tight for $\pst$ in $\R$. Setting 
\begin{equation}
V_{s,t}\ph(x)=\ph(U(t,s)x),\ \ \ph\in C_b(X),\ x\in X,\ s\leq t,
\end{equation}
we show first that there exist at least two evolution system of measures for $V_{s,t}$ in $\R$. We claim that the families $\{\mu^{(1)}_t\}_{t\in\R}$ and $\{\mu^{(2)}_t\}_{t\in\R}$ defined by $\mu^{(1)}_t\equiv\delta_0$ and $\displaystyle{\mu_t^{(2)}=\delta_{\frac{e_1}{m_t}}}$ with $\displaystyle{m_t=e^{\int_{-\infty}^ta_1(\tau)\,d\tau}}$ for all $t\in\R$, are evolution system of measures for $V_{s,t}$ in $\R$.
Indeed given $\ph\in C_b(X)$, we have
\begin{align}
&\int_X\ph\left(U(t,s)x\right)\,\delta_0(dx)=\ph\left(U(t,s)0\right)=\ph(0)=\int_X\ph\left(x\right)\,\delta_0(dx),\nonumber \\
&\int_X\ph\left(U(t,s)x\right)\,\delta_{\frac{e_1}{m_t}}(dx)=\ph\left(U(t,s)\frac{e_1}{m_t}\right)=\ph\left(\frac{e_1}{m_s}\right)=\int_X\ph\left(x\right)\,\delta_{\frac{e_1}{m_s}}(dx).\nonumber
\end{align}
Hence, by Theorem \ref{esistenza}, $\{\g_t\}_{t\in\R}$ with $\g_t$ given by \eqref{gammat} and $\{\nu_t\}_{t\in\R}$ with $\nu_t=\g_t\star\mu_t^{(2)}$ are evolutions systems of measures for $\pst$ in $\R$.
\end{oss}

\subsection{A non autonomous version of the classical Ornstein-Uhlenbeck operator}

Let $A(t)=a(t) I$, where $a$ is a continuous and bounded real valued map on $\R$ and set $\displaystyle{\sup_{t\in \R}a_0(t)=a_0}$. 
Hence $$U(t,s)=\exp\biggl(\int_s^ta(\tau)\,d\tau\biggr) I,\ \ (s,t)\in\R^2$$ is continuous with values in $\op(X)$ and it is associated to the family $\{A(t)\}_{t\in\R}$. 

Let $\{B(t)\}_{t\in\R}\subseteq\op(X)$ be a family of operators satisfying Hypothesis \ref{1} (2).
Since
\begin{equation}
Q(t,s)=\int_s^t \exp\biggl(2\int_\si^ta(\tau)\,d\tau\biggr)Q(\si)\,d\si\ \ (s,t)\in\Delta,\nonumber
\end{equation}
where $Q(\si)=B(\si)B(\si)^\star$, a sufficient and obvious condition for $\Tr{Q(t,s)}<+\infty$ is that $\Tr{Q(\si)}<+\infty$ for a.e. $\si\in\R$ and $\si\longmapsto\Tr{Q(\si)}\in L^1(\R)$. 
Indeed, if $\{e_k\}_{k\in\Nat}$ is a Hilbert basis of $X$, we have
\begin{align}\label{tracciamallia}
\ix{Q(t,s)e_k}{e_k}\leq e^{2a_0(t-s)}\int_s^t \ix{Q(\si)e_k}{e_k}\,d\si\ \ (s,t)\in\Delta.
\end{align}
In this case, $\pst$ is a non autonomous generalization of  the classical Ornstein-Uhlenbeck semigroup widely used in the Malliavin Calculus. 

We note that $D(A(t))=D(A(t)^\star)=X$ for all $t\in\R$ and Hypothesis \ref{2bis} is obviously satisfied.

In addition to the above assumptions on the trace of the operators $Q(\si)$, we require that for all $t\in\R$ there exists $C_t>0$ such that 
\begin{equation}\label{normbt}
\norm{B(s)}_{\op(X)}\leq C_t\norm{B(t)}_{\op(X)},\quad \forall\,s<t.
\end{equation}
Moreover we assume also that $a_0< 0$.

By \eqref{tracciamallia} and  \eqref{normbt}, \eqref{omega} holds and $$\norm{\uts}_{\op(X)}\leq e^{a_0(t-s)},\ \ (s,t\in\R^2).$$ For all $(s,t)\in\Delta$ we have
\begin{align}
\ix{Q(t,s)e_k}{e_k}&\leq e^{2a_0(t-s)}\int_s^t \ix{Q(\si)e_k}{e_k}\,d\si= e^{2a_0(t-s)}\int_s^t \norm{B(\si)^\star e_k}_X^2\,d\si\nonumber\\
&\leq e^{2a_0(t-s)}(t-s) C_t^2 \norm{B(t)^\star e_k}_X^2=C_t^2 e^{2a_0(t-s)}(t-s)\ix{Q(t)e_k}{e_k}.
\end{align} 
Therefore $$\sup_{s<t}\Tr{Q(t,s)}<+\infty$$ and by Theorem \ref{unicità} there exists a unique evolution system of measures for $\pst$ in $\R$.

Moreover \eqref{cond-logsob2} holds. Indeed
\begin{align}
\norm{\Qp{s}}_{\op(X)}^2=\norm{B(s)^\star}^2_{\op(X)}\leq C_t^2\norm{B(t)^\star}^2_{\op(X)}= C_t^2\norm{\Qp{t}}_{\op(X)}^2.
\end{align}
Then, by Proposition B.1 in Appendix B in \cite[pag. 429]{MR3236753}, $H_s\subseteq H_t$ with continuous embedding and for $x\in H_s$ we have
\begin{align}\label{stimamalli}
\norm{U(t,s)x}_{H_t}=\norm{\exp\biggl(\int_s^ta(\tau)\,d\tau\biggr)x}_{H_t}\leq C_t e^{a_0(t-s)}\norm{x}_{H_s}.
\end{align}
Since $a_0<0$, the hypotheses of Theorems \ref{LOG} and \ref{HYPER} are satisfied.

\begin{oss}\label{ottimale}
If $a(t)=-1$ and ${\rm Range}\left(Q(t)^{1/2}\right)={\rm Range}\left(Q(0)^{1/2}\right)$ for every $t\in\R$, then $a_0=-1$ and $c_t=1$ in \eqref{stimamalli} so $\kappa$ defined in \eqref{logsobcost} is equal to $\displaystyle{\frac{1}{2}}$. Moreover \eqref{iper} holds for every $(s,t)\in\Delta$ and $p\leq C(t,s,q):=(q-1)e^{t-s}+1$. We stress that $C(t,s,q)$ is non autonomous version of the optimal constant given in \cite[Rmk.\,3.4]{gro75} and \cite[p. 242 under (1.3)]{fur1998}. 

We now give an explicit example where ${\rm Range}(B(t))={\rm Range}(B(0))$ for every $t\in\R$. Let $d\in\Nat$, let $X=L^2(\os)$ where $\os$ is a bounded open subset of $\R^d$ with smooth enough boundary.  We choose $B(t)=(A(t))^{-\g}$ where $A(t)$ is the realization of $\oa(t)$ given by \eqref{operatoreat} with Dirichlet boundary conditions, and $\g$ satisfies \eqref{sceltagamma2}.
\end{oss}

\section*{Acknowledgement}
The authors are members of GNAMPA (Gruppo Nazionale per l'Analisi Matematica, la Probabilit\`a le loro Applicazioni) of the Italian Istituto Nazionale di Alta Matematica (INdAM).

\addcontentsline{toc}{section}{References}
\printbibliography

\end{document}